\documentclass{article}
\usepackage[scale=0.8]{geometry}

\usepackage{amsmath}
\usepackage{amssymb}
\usepackage{amsthm}
\usepackage{dsfont}
\usepackage{amsfonts}
\usepackage{cases}
\usepackage{url}
\theoremstyle{plain}
\newtheorem{theor}{Theorem}[section]
\newtheorem{lem}[theor]{Lemma}
\newtheorem{prop}[theor]{Proposition}
\newtheorem{cor}[theor]{Corollary}

\theoremstyle{definition}

\newtheorem{rems}[theor]{Remarks}

\newtheorem{rem}[theor]{Remark}

\usepackage[applemac]{inputenc}
\usepackage[polutonikogreek,english,frenchb]{babel}
\usepackage[T1]{fontenc}
\usepackage{csquotes}
\usepackage[toc,page]{appendix}
\usepackage{textcomp}
\usepackage{amscd}
\usepackage{graphicx}
\usepackage{enumerate}
\usepackage{cancel}
\usepackage{extpfeil}
\usepackage[all]{xy}
\usepackage{empheq}
\usepackage{esint}
\usepackage{mathtools}
\usepackage{float}
\usepackage{wrapfig}

\newcommand{\e}{\epsilon}

\newcommand{\R}{\mathbb R}

\newcommand{\B}{\mathcal B}

\newcommand{\Pc}{\mathcal P}
\newcommand{\Ec}{\mathcal E}

\newcommand{\pv}{\operatorname{p.v.}}

\newcommand{\Id}{\operatorname{Id}}

\newcommand{\Div}{\operatorname{div}}

\newcommand{\loc}{{\operatorname{loc}}}

\newcommand{\Ld}{\operatorname{L}}
\newcommand{\Bd}{\operatorname{B}}

\usepackage[toc,page]{appendix}

\newcommand{\cvf}[1]{\mathrel{\mathop{\xrightharpoonup{#1}}}}

\numberwithin{equation}{section}

\usepackage{color}

\usepackage[colorlinks,citecolor=black,urlcolor=black]{hyperref}

\title{Mean-field limits for some Riesz interaction gradient flows}
\author{Mitia Duerinckx}
\date{}

\begin{document}
\selectlanguage{english}
\maketitle

This paper is concerned with the mean-field limit for the gradient flow evolution of particle systems with pairwise Riesz interactions, as the number of particles tends to infinity. Based on a modulated energy method, using regularity and stability properties of the limiting equation,
as inspired by the work of Serfaty~\cite{Serfaty-15} in the context of the Ginzburg-Landau vortices, we prove a mean-field limit result in dimensions $1$ and $2$ in cases for which this problem was still open.

\bigskip
\section{Introduction}
We consider the energy of a system of $N$ particles in the Euclidean space $\R^d$ ($d\ge1$) interacting via (repulsive) Riesz pairwise interactions:
\[H_N(x_1,\ldots,x_N)=\sum_{i\ne j}^Ng_s(x_i-x_j),\qquad x_i\in\R^d,\]
where the interaction kernel is given by
\[g_s(x):=\begin{cases}c_{d,s}^{-1}|x|^{-s},&\text{if $0<s<d$};\\-c_{d,0}^{-1}\log(|x|),&\text{if $s=0$};\end{cases}\]
with $c_{d,s}>0$ some normalization constants. We note that the Coulomb case corresponds to the choice $s=d-2$, $d\ge2$. Particle systems with more general Riesz interactions as considered here are extensively motivated in the physics literature (cf.~for instance \cite{BBDR-05,Mazars-11}), as well as in the context of approximation theory with the study of Fekete points (cf.~\cite{Hardin-Saff-04} and the references therein). Recently, a detailed description of such systems beyond the mean-field limit in the static case was obtained in~\cite{Petrache-Serfaty-14}, and also in~\cite{Leble-Serfaty-15} for non-zero temperature. In the present contribution, we are rather interested in the dynamics of such systems, and more precisely in a rigorous justification of the mean-field limit of their gradient flow evolution as the number $N$ of particles tends to infinity, which has indeed remained an open problem whenever $s\ge d-2$, $s>0$, $d\ge2$.

We thus consider the trajectories $x_{i,N}^t$ driven by the corresponding flow, i.e. the solutions to the following system of ODEs:
\begin{align}\label{eq:vortexdyn}
\partial_t x_{i,N}^t=-\frac1N\nabla_iH_N(x_{1,N}^t,\ldots,x_{N,N}^t),\qquad x_{i,N}^t|_{t=0}=x_{i,N}^\circ,\qquad i=1,\ldots,N,
\end{align}
where $(x_{i,N}^\circ)_{i=1}^N$ is a sequence of $N$ distinct initial positions. Since energy can only decrease in time and since the interaction is repulsive, particles cannot collide, and moreover it is easily seen that a particle cannot escape to infinity in finite time; from these observations and from the Picard-Lindelöf theorem, we may conclude that the trajectories $x_{i,N}^t$ are smooth and well-defined on the whole of $\R^+:=[0,\infty)$. As the number of particles gets large, we would naturally like to pass to a continuum description of the system, in terms of the particle density distribution.
For that purpose, we define the empirical measure associated with the point-vortex dynamics:
\begin{align}\label{eq:empiricalmeas}
\mu_N^t:=\frac1N\sum_{i=1}^N\delta_{x_{i,N}^t},
\end{align}
and the question is then to understand the limit of $\mu_N^t$ as $N\uparrow\infty$. More precisely, assuming convergence at initial time
\[\mu_N^\circ:=\frac1N\sum_{i=1}^N\delta_{x_{i,N}^\circ}\cvf*\mu^\circ,\qquad\text{as $N\uparrow\infty$,}\]
formal computations lead us to expect under fairly general assumptions $\mu_N^t\cvf*\mu^t$ for all $t\ge0$, where $\mu^t$ is a solution to the following nonlocal nonlinear PDE on $\R^+\times\R^d$, sometimes called the {\it fractional porous medium equation}:
\begin{align}\label{eq:pdelim}
\partial_t\mu^t=\Div(\mu^t\nabla h^t),\qquad h^t:=
g_s\ast\mu^t,\qquad\mu^t|_{t=0}=\mu^\circ.
\end{align}
This equation in the weak sense just means the following: $\mu\in \Ld^1_{\loc}(\R^+;\Ld^1(\R^d))$, $g_s\ast\mu\in \Ld^1_\loc(\R^+;W^{1,1}_{\loc}(\R^d))$, $\mu\nabla g_s\ast \mu\in \Ld^1_{\loc}(\R^+;\Ld^1(\R^d))$, and, for all $\phi\in C^\infty(\R^+;C^\infty_c(\R^d))$ such that $\phi(t,\cdot)=0$ for all $t>0$ large enough,
\[\int_{\R^+}\int_{\R^d}\mu^t(x)(\partial_t\phi(t,x)-\nabla \phi(t,x)\cdot\nabla g_s\ast \mu^t(x))dxdt+\int_{\R^d}\mu^\circ(x)\phi(0,x)dx=0.\]
As far as existence issues as well as basic properties of the solutions of~\eqref{eq:pdelim} are concerned, we refer to~\cite{Caffarelli-Vazquez-11,Caffarelli-Soria-Vasquez-13} for $d-2<s<d$, $s\ge0$, to~\cite{Lin-Zhang-00,Du-Zhang-03,Ambrosio-Serfaty-08,Serfaty-Vasquez-14} for $s=d-2$, $d\ge2$, and to~\cite{Carrillo-Choi-Hauray-14} for $0\le s<d-2$.
See also Proposition~\ref{prop:existence} below.

In the case of logarithmic interactions $s=0$, this expected mean-field limit result was essentially first proven (in arbitrary dimension) by Schochet~\cite{Schochet-96} based on his simplification~\cite{Schochet-95} of the proof of Delort's theorem~\cite{Delort-91} on existence of weak solutions to the 2D Euler equation with initial nonnegative vorticity in $H^{-1}$. Schochet's original paper~\cite{Schochet-96} was actually only concerned with the mean-field limit for a particle approximation of the 2D Euler equation, but the same argument directly applies to the present setting. However, due to a possible lack of uniqueness of $\Ld^1$ weak solutions to equation~\eqref{eq:pdelim}, Schochet~\cite{Schochet-96} could only prove that the empirical measure $\mu_N^t$ converges up to a subsequence to {\it some} solution of~\eqref{eq:pdelim}.
The key idea, which only holds for logarithmic interactions, consists in exploiting some logarithmic gain of integrability to find uniform bounds on the number of close particles, which allows to directly pass to the limit in the equation and conclude by a compactness argument.

In the case $0\le s<d-2$, $d\ge3$, the complete mean-field limit result (not restricted to a subsequence) was proven more recently by Hauray~\cite{Hauray-09} (see also~\cite{Carrillo-Choi-Hauray-14}), but his method, based on a control of the infinite Wasserstein distance, cannot be adapted at all to higher powers $s\ge d-2$. In the 1D case, Berman and \"Onnheim~\cite{Berman-Onnheim} obtained a similar result for the whole range $0\le s<1$, in the framework of Wasserstein gradient flows, but their method cannot be extended to higher dimensions.

Very recently, in the context of the 2D Gross-Pitaevskii and parabolic Ginzburg-Landau equations, Serfaty~\cite{Serfaty-15} proposed a new way of proving such mean-field limits\footnote{In~\cite{Serfaty-15}, the questions are different in nature, since they consist in passing to the limit in PDE evolutions, but are similar in spirit since one wishes to understand the limiting dynamics of point vortices which essentially behave like Coulomb-interacting particles.}, based on a Gronwall argument for the so-called {\it modulated energy}, which is some adapted measure of the distance to the (postulated) limit.
This idea of proof
originates in the relative entropy method first introduced by Yau~\cite{Yau-91} for hydrodynamic limits (see e.g.~\cite{StRaymond-09} and the references therein for later developments), the modulated Hamiltonian method used by Grenier~\cite{Grenier-99} for boundary layer problems, and the modulated energy method designed by Brenier~\cite{Brenier-00} for the quasi-neutral limit of the Vlasov-Poisson system.
The advantage of this method is to be completely global, bypassing the need for a precise understanding of the microscopic dynamics. It relies on the regularity of the solution to the limiting equation, and exploits its stability properties. As will be seen, however, we are able to apply this method in the present context only in dimensions $1$ and $2$ and for $s$ not too large. More precisely, we treat in 1D the whole range $0\le s<1$ as in~\cite{Berman-Onnheim}, while in 2D we treat but the case $0\le s<1$, which is new and in particular completes Schochet's partial result~\cite{Schochet-96} in the logarithmic case.
Our main result, for which we need an additional regularity assumption on the limiting equation (cf.~Remark~\ref{rem}(a) below),
is as follows:

\begin{theor}\label{th:mfl}
Let $d=1$ or $2$, and $0\le s<1$. Let $\mu^\circ$ be a probability measure, and in the case $s=0$ also assume $\int_{\R^d}\log(2+|x|)|\mu^\circ(x)|dx<\infty$. Assume that equation~\eqref{eq:pdelim} admits a solution $\mu^t$ that belongs to $\Ld^\infty([0,T]; C^{\sigma}(\R^d))$ for some $T>0$ and some $\sigma>2-d+s$. In the case $s=0$, $d=1$, also assume $\nabla \mu\in\Ld^\infty([0,T];\Ld^p(\R^d))$ for some $p<\infty$. Let $\mu_N^\circ\cvf{*}\mu^\circ$ as above,
assume the convergence of the initial energy
\begin{align}\label{eq:asboundenergy}
\lim_{N\uparrow\infty}\frac1{N^2}H_N(x_{1,N}^\circ,\ldots,x_{N,N}^\circ)=\int_{\R^d}\int_{\R^d} g_s(x-y)d\mu^\circ(x)d\mu^\circ(y)<\infty,
\end{align}
and let $\mu_N^t$ be defined by~\eqref{eq:vortexdyn}--\eqref{eq:empiricalmeas}.
Then $\mu^t$ is the only weak solution to~\eqref{eq:pdelim} up to time $T$ in $\Ld^\infty([0,T];\Ld^\infty(\R^d))$, and for all $t\in[0,T]$ we have $\mu_N^t\cvf*\mu^t$, as well as the convergence of the energy
\[\lim_{N\uparrow\infty}\frac1{N^2}H_N(x_{1,N}^t,\ldots,x_{N,N}^t)=
\int_{\R^d}\int_{\R^d} g_s(x-y)d\mu^t(x)d\mu^t(y)<\infty.\]
\end{theor}

\begin{rems}$   $\label{rem}\begin{enumerate}[(a)]
\item For a compactly supported probability measure $\mu^\circ\in\Ld^\infty(\R^d)$, equation~\eqref{eq:pdelim} always admits a solution in $\Ld^\infty(\R^+;\Ld^\infty(\R^d))$, which remains a compactly supported probability measure for all times (see Proposition~\ref{prop:existence} below). As far as the additional regularity assumption is concerned, as explained in Section~\ref{chap:propeqn}, it has been proven to hold with $T=\infty$ in the case $0\le s\le d-2$, $d\ge2$, and at least up to some time $T>0$ in the case $d-2<s\le d-1$, $s\ge0$, for sufficiently smooth initial data $\mu^\circ$,
but all other cases remain unsolved, and the additional regularity assumption in Theorem~\ref{th:mfl} above is then crucially needed.\\
In the 2D Coulomb case $s=0$, $d=2$, as this regularity problem is solved (cf. \cite[Theorem~1]{Lin-Zhang-00}), the conclusion of Theorem~\ref{th:mfl} above holds in that case with $T=\infty$ (that is, with $\Ld^\infty([0,T];\cdot)$ replaced by $\Ld^\infty_\loc(\R^+;\cdot)$), whenever we have $\mu^\circ\in\Pc(\R^2)\cap C^{\sigma}(\R^2)$ for some $\sigma>0$, and $\int_{\R^2}\log(2+|x|)|\mu^\circ(x)|dx<\infty$. This completes Schochet's partial result~\cite{Schochet-96}. As the regularity problem is further solved in short time in Sobolev spaces in the case $d-2<s\le d-1$, $s\ge0$ (cf.~\cite{Zhou-16}), the conclusion of Theorem~\ref{th:mfl} above holds for some $T>0$ (depending on initial data) in the case $0<s\le 1$, $d=2$, whenever $\mu^\circ\in \Pc(\R^2)\cap H^{\sigma}(\R^2)$ for some $\sigma>2$, and also in the 1D logarithmic case $s=0$, $d=1$, whenever $\mu^\circ\in \Pc(\R)\cap H^{\sigma}(\R)$ for some $\sigma>3/2$, and $\int_{\R}\log(2+|x|)|\mu^\circ(x)|dx<\infty$.
\item A closer look at the proof actually shows the following quantitative statement, where the distance between $\mu_N^t$ and $\mu^t$ is measured in terms of the modulated energy: for all $t\in[0,T]$, we have for some $C_t>0$
\begin{align*}
&\iint_{x\ne y}g_s(x-y)d(\mu_N^t-\mu^t)(x)d(\mu_N^t-\mu^t)(y)\\
\le~&C_t\iint_{x\ne y}g_s(x-y)d(\mu_N^\circ-\mu^\circ)(x)d(\mu_N^\circ-\mu^\circ)(y)+ C_t \begin{cases}N^{-\frac{(1-s)(1-\sigma)}{1+s-\sigma}},&\text{if $s>0$};\\ N^{-1}\log N,&\text{if $s=0$}.\end{cases}
\end{align*}
\item The well-preparedness assumption~\eqref{eq:asboundenergy} for the initial positions $(x_{i,N}^\circ)_{i=1}^N$, $N\ge0$, is statistically relevant, in the sense that it is automatically satisfied almost surely if, for each $N$, the positions $(x_{i,N}^\circ)_{i=1}^N$ are chosen to be independent $\R^d$-valued random variables with law $\mu^\circ$. This easily follows from the strong law of large numbers, together with the bound (for $s>0$)
\begin{align*}
\iint |g_s(x-y)|d\mu^\circ(x)d\mu^\circ(y)&\lesssim\iint_{|x-y|\le1} |x-y|^{-s}d\mu^\circ(x)d\mu^\circ(y)+\iint_{|x-y|>1} d\mu^\circ(x)d\mu^\circ(y)\lesssim \|\mu^\circ\|_{\Ld^\infty}+1.
\end{align*}
\item\label{eq:+pot} We may also add to the energy a potential $V$, thus turning $H_N$ into $\tilde H_N:=H_N+V$. If $V\in C^2(\R^d)$ satisfies $\|\nabla^2V\|_{\Ld^\infty}<\infty$, then all the arguments may be directly adapted, as long as the corresponding limit equation
\[\partial_t\mu^t=\Div(\mu^t\nabla (h^t+V)),\qquad h^t:=g_s\ast \mu^t,\qquad\mu^t|_{t=0}=\mu^\circ,\]
admits a regular enough solution.
\item In dimension $d=2$, we could also consider a mix between the gradient flow~\eqref{eq:vortexdyn} and its conservative counterpart, that is, replacing~\eqref{eq:vortexdyn} by the following system of ODEs, for $i=1,\ldots,N$,
\[\partial_tx_{i,N}^t=-\frac\alpha N\nabla_i H_N(x_{1,N}^t,\ldots,x_{N,N}^t)-\frac\beta N\nabla_i^\bot H_N(x_{1,N}^t,\ldots,x_{N,N}^t)-\nabla V(x_{i,N}^t),\qquad x_{i,N}^t|_{t=0}=x_{i,N}^\circ,\]
where we have also added a potential $V$ as in item~\eqref{eq:+pot} above. If $\alpha>0$, then all the arguments may again be directly adapted, as long as the corresponding limit equation
\[ \partial_t\mu^t=\Div(\mu^t(\alpha\nabla h^t+\beta\nabla^\bot h^t+\nabla V)),\qquad h^t:=g_s\ast\mu^t,\qquad \mu^t|_{t=0}=\mu^\circ,\]
admits a regular enough solution. (Note that the same proof can a priori not work for the choice $\alpha=0$, since in Step~2 of the proof of Proposition~\ref{prop:MFL} below some term cannot be estimated directly and needs instead to be absorbed using the negative diffusion term, which would however vanish in the case $\alpha=0$.)
\end{enumerate}\end{rems}

\medskip
\noindent{\bf Notation.}
Denote by $B(x,r)$ the ball of radius $r$ centered at $x$ in $\R^d$, and set $B_r:=B(0,r)$. We also use the notation $a\wedge b=\min\{a,b\}$ and $a\vee b=\max\{a,b\}$ for all $a,b\in\R$.
The space of probability measures on $\R^d$ is denoted by $\Pc(\R^d)$, and, for all $\sigma>0$, $C^\sigma(\R^d)$ stands as usual for the Hölder space $C^{\lfloor \sigma\rfloor,\sigma-\lfloor \sigma\rfloor}(\R^d)$, while $C_c^\sigma(\R^d)$ denotes the subspace of compactly supported functions.
In the sequel, $C$ denotes any positive constant only depending on $d$ and $s$.
We denote by $C_t$ any positive constant only depending on $d$, $s$ and on time $t$, such that $C_t\le C_T$ for all $t\in[0,T]$ and all $T>0$. We also denote $x\lesssim y$ and $x\lesssim_ty$ for $x\le Cy$ and $x\le C_ty$, respectively, and we use the notation $x\simeq y$ if both $x\lesssim y$ and $y\lesssim x$ hold. Finally, we denote by $o_a(1)$ a quantity that goes to~$0$ when the parameter $a$ goes to its limit, uniformly with respect to other parameters, and we write $o_a^{(b)}(1)$ if it converges to $0$ only for any {\it fixed} value of the parameter $b$.

\section{Proof of Theorem~\ref{th:mfl}}

\subsection{Strategy of the proof}

Translating the idea of~\cite{Serfaty-15} in the present setting (see also~\cite{Brenier-00}), the clue of the proof of Theorem~\ref{th:mfl} comes from the following stability estimate, which we first present for simplicity in the Coulomb case.

\begin{lem}[Stability --- Coulomb case]\label{lem:proofsmooth}
Let $s=d-2, d\ge2$. Let $\mu_1^\circ,\mu_2^\circ\in\Pc(\R^d)\cap\Ld^\infty(\R^d)$, and in the case $d=2$ also assume $\int_{\R^2}\log(2+|x|)(|\mu_1^\circ(x)|+|\mu_2^\circ(x)|)dx<\infty$. For $i=1,2$, let $\mu_i^t$ be a weak solution of equation~\eqref{eq:pdelim} with initial condition $\mu_i^\circ$, denote $h_i^t:=g_{d-2}\ast \mu_i^t$, and assume $\mu_1,\mu_2\in \Ld^\infty([0,T];\Ld^\infty(\R^d))$ and $\nabla^2h_2\in\Ld^1([0,T];\Ld^\infty(\R^d))$ for some $T>0$. Then, for all $t\in[0,T]$,
\begin{align}\label{eq:proofsmooth0}
&\int_{\R^d}\int_{\R^d} g_{d-2}(x-y)d(\mu_1^t-\mu_2^t)(x)d(\mu_1^t-\mu_2^t)(y)\\
\le~& e^{C\int_0^t\|\nabla^2h_2^u\|_{\Ld^\infty}du}\int_{\R^d}\int_{\R^d} g_{d-2}(x-y)d(\mu_1^\circ-\mu_2^\circ)(x)d(\mu_1^\circ-\mu_2^\circ)(y).\nonumber
\end{align}
\end{lem}

\begin{proof}
Proposition~\ref{prop:existence}(ii) below yields $\nabla (h_1-h_2)\in\Ld^\infty(\R^+;\Ld^2(\R^d))$. Combining this with the additional boundedness assumptions, all integration by parts arguments in the sequel may be justified.

Using the equations for $\mu_1^t$ and $\mu_2^t$, the time derivative of the left-hand side of~\eqref{eq:proofsmooth0} can be computed as follows:
\begin{align}
&\partial_t\int_{\R^d}\int_{\R^d} g_{d-2}(x-y)d(\mu_1^t-\mu_2^t)(x)d(\mu_1^t-\mu_2^t)(y)\nonumber\\
=~&2\int_{\R^d}(h_1^t-h_2^t)( \partial_t\mu_1^t-\partial_t\mu_2^t)=-2\int_{\R^d} \nabla (h_1^t-h_2^t) (\mu_1^t\nabla h_1^t- \mu_2^t\nabla h_2^t)\nonumber\\
=~&-2\int_{\R^d} |\nabla (h_1^t-h_2^t)|^2\mu_1^t-2\int_{\R^d} \nabla h_2^t\cdot \nabla (h_1^t-h_2^t)\, (\mu_1^t- \mu_2^t).\label{eq:proofsmooth}
\end{align}
The first term in the right-hand side is nonpositive, so it suffices to estimate the second one. Using the relations $-\Delta h_i^t=\mu_i^t$, $i=1,2$ (which hold with a unit factor for the suitable choice of the normalization constant $c_{d,d-2}>0$), the product $\nabla (h_1^t-h_2^t)\, (\mu_1^t- \mu_2^t)$ may be rewritten à la Delort using the stress-energy tensor:
\begin{align}
-2\nabla (h_1^t-h_2^t)\, (\mu_1^t- \mu_2^t)&=2\nabla (h_1^t-h_2^t)\,\Delta (h_1^t-h_2^t)\nonumber\\
&=\Div\left(2\nabla(h_1^t-h_2^t)\otimes\nabla(h_1^t-h_2^t)-\Id|\nabla (h_1^t-h_2^t)|^2\right),\label{eq:delortform}
\end{align}
where the divergence of a $2$-tensor here denotes the vector whose coordinates are the divergences of the corresponding columns of the tensor. Combining this with an integration by parts, we find
\begin{align*}
2\int_{\R^d} \nabla h_2^t\cdot \nabla (h_1^t-h_2^t)\, (\mu_1^t- \mu_2^t)&=-\int_{\R^d} \Div\left(2\nabla(h_1^t-h_2^t)\otimes\nabla(h_1^t-h_2^t)-\Id|\nabla (h_1^t-h_2^t)|^2\right)\cdot\nabla h_2^t\\
&=\int_{\R^d} \left(2\nabla(h_1^t-h_2^t)\otimes\nabla(h_1^t-h_2^t)-\Id|\nabla (h_1^t-h_2^t)|^2\right):\nabla^2 h_2^t.
\end{align*}
The inequality $2|ab|\le a^2+b^2$ and an integration by parts then yield
\begin{align}
\bigg|\int_{\R^d} \nabla h_2^t\cdot \nabla (h_1^t-h_2^t)\, (\mu_1^t- \mu_2^t)\bigg|&\lesssim \|\nabla^2 h_2^t\|_{\Ld^\infty} \int_{\R^d} |\nabla (h_1^t-h_2^t)|^2\nonumber\\
&= \|\nabla^2 h_2^t\|_{\Ld^\infty} \int_{\R^d}  (h_1^t-h_2^t) (\mu_1^t-\mu_2^t)\nonumber\\
&=\|\nabla^2 h_2^t\|_{\Ld^\infty} \int_{\R^d}\int_{\R^d} g_{d-2}(x-y)d(\mu_1^t-\mu_2^t)(x)d(\mu_1^t-\mu_2^t)(y),\label{eq:ippCoul}
\end{align}
so that the result~\eqref{eq:proofsmooth0} follows from~\eqref{eq:proofsmooth} and a Gronwall argument.
\end{proof}

We are also interested here in the non-Coulomb case $d-2<s<d$, $s\ge0$, and hence, just as in~\cite{Petrache-Serfaty-14}, we need to use the extension method of Caffarelli and Silvestre~\cite{Caffarelli-Silvestre-07} (cf. Section~\ref{chap:caffsilv} below) in order to find a similar Delort-type formula as in~\eqref{eq:delortform} of the proof above, and then repeat the same integration by parts argument, thus circumventing the fact that the Riesz kernel is not the convolution kernel of a local operator. This allows to prove the same estimate as above in all cases $0\vee(d-2)\le s<d$ with $g_{d-2}$ replaced by $g_s$ (cf. Lemma~\ref{lem:proofsmooth2} below).

This stability estimate gives us a control of the $H^{-1}$-distance (or the $H^{-(d-s)/2}$-distance, for general $0\le s<d$) between $\mu_1^t$ and $\mu_2^t$ in terms of the initial distance, up to a factor that only depends on the regularity of $\mu_2^t$ in the form of $\|\nabla^2 h_2^t\|_{\Ld^\infty}$. We would then like to replace $\mu_2^t$ by the smooth solution $\mu^t$ and to replace $\mu_1^t$ by $\mu_N^t$. However, the corresponding distance would then be infinite because of the presence of Dirac masses, and moreover $\mu_N^t$ does not exactly satisfy the limiting equation~\eqref{eq:pdelim}. The idea of the proof of Theorem~\ref{th:mfl} consists in finding a suitable way of adapting the proof above to that setting.

First, the natural way of giving a meaning to this divergent distance between $\mu_N^t$ and $\mu^t$ simply consists in excluding the diagonal terms, thus considering
the (renormalized) {\it modulated energy}
\[\Ec_N(t)=\iint_{x\ne y}g_s(x-y)d(\mu_N^t-\mu^t)(x)d(\mu_N^t-\mu^t)(y).\]
The goal is then to compute the time-derivative $\partial_t\Ec_N(t)$, and trying to adapt the proof of the stability estimate above to bound it by $C\Ec_N(t)$ for some constant $C>0$, up to a vanishing additive error. However, at the end of the proof above, the use of the inequality $2|ab|\le a^2+b^2$ is clearly not compatible with the removal of the diagonal terms. To solve this main issue, the crucial idea is due to Serfaty~\cite{Serfaty-15}: regularizing the Dirac masses at a (small) scale $\eta$ so that the diagonal terms become well-defined and diverge only as $\eta\downarrow0$, we need to try to construct around the particle locations small balls that contain most of the divergent $\eta$-approximate energy, so that excluding diagonal terms essentially amounts to restricting the $\eta$-approximate integrals to outside these small balls. Using the same approximation argument as in~\cite{Serfaty-15} to be allowed to restrict all integrals to outside these balls, the end of the proof above is then easily adapted, using the inequality $2|ab|\le a^2+b^2$ only on the restricted domain.

In this way, for any $0\vee(d-2)\le s<d$, we manage to prove $\partial_t\Ec_N(t)\le C\Ec_N(t)+o_N(1)$ under some mesoscopic regularity assumption on the distribution of the particles in time (cf. Proposition~\ref{prop:MFL} below). Finally, in the case $s<1$ (hence our limitation to that regime), these conditions can be directly checked using a modification of the ball construction introduced by~\cite{Sandier-98,Jerrard-99} for the analysis of the Ginzburg-Landau vortices (cf. Section~\ref{chap:ball0} below).


\subsection{Extension representation for fractional Laplacian}\label{chap:caffsilv}

We recall here the extension representation for the fractional Laplacian by Caffarelli and Silvestre~\cite{Caffarelli-Silvestre-07} (we follow notation of \cite[Section~1.2]{Petrache-Serfaty-14}). Let $0\vee(d-2)< s<d$ be fixed. For a finite Borel measure $\mu$ on $\R^d$, the associated Riesz potential $(-\Delta)^{-(d-s)/2}\mu$ can be written (up to a constant) as
\begin{align}\label{eq:defriesz}
h^\mu(x):=c_{d,s}^{-1}\int_{\R^d}|x-z|^{-s}d\mu(z)=g_s\ast\mu(x).
\end{align}
We denote coordinates in $\R^d\times\R$ by $(x,\xi)$, and we denote by $\mu\delta_{\R^d\times\{0\}}$ the Borel measure on $\R^d\times\R$ defined as follows: for all $\phi\in C^\infty_c(\R^d\times\R)$,
\[\int_{\R^d\times\R} \phi(x,\xi) d(\mu\delta_{\R^d\times\{0\}})(x,\xi):=\int_{\R^d} \phi(x,0) d\mu(x).\]
Extending $h^\mu$ to $\R^{d}\times\R$ via
\begin{align*}
h^\mu(x,\xi):=\int_{\R^{d}}|(x,\xi)-(z,0)|^{-s}d\mu(z)=g_s\ast (\mu\delta_{\R^d\times\{0\}})(x,\xi),
\end{align*}
where we denote $g_s(x,\xi)=c_{d,s}^{-1}|(x,\xi)|^{-s}$, and choosing $\gamma:=s+1-d\in(-1,1)$, the extended function $h^\mu$ on $\R^{d}\times\R$ satisfies in the distributional sense
\[-\Div(|\xi|^{\gamma}\nabla h^\mu)=\mu\delta_{\R^d\times\{0\}}.\]
The function $g_s$ is indeed a fundamental solution of the operator $-\Div(|\xi|^\gamma\nabla)$ on $\R^{d}\times\R$, in the sense that $-\Div(|\xi|^\gamma\nabla g_s)=\delta_0$ on $\R^d\times\R$. The normalization constant $c_{d,s}$ is chosen exactly to satisfy this property with a unit factor (for an explicit value, see Step~1 of the proof of Lemma~\ref{lem:comput} below).

In the case $s=0$, $d=1$, denoting $g_0(x,\xi)=-c_{d,0}^{-1}\log(|(x,\xi)|)$, we have $-\Delta g_0=\delta_0$ on the extended space $\R\times\R$, for the suitable choice of the normalization constant $c_{1,0}$, so the above again holds with $\gamma=0=s+1-d$.
(In the Coulomb case $s=d-2$, $d\ge2$, no extension is needed, and the normalization $c_{d,d-2}$ is simply chosen such that $-\Delta g_{d-2}=\delta_0$ on $\R^d$.)

Using this extension representation, we may now directly adapt the stability estimate of Lemma~\ref{lem:proofsmooth} to the non-Coulomb case:

\begin{lem}[Stability --- Riesz case]\label{lem:proofsmooth2}
Let $0\vee(d-2)\le s<d$. Let $\mu_1^\circ,\mu_2^\circ\in\Pc(\R^d)\cap\Ld^\infty(\R^d)$, and in the case $s=0$ also assume $\int_{\R^d}\log(2+|x|)(|\mu_1^\circ(x)|+|\mu_2^\circ(x)|)dx<\infty$. For $i=1,2$, let $\mu_i^t$ be a weak solution of equation~\eqref{eq:pdelim} with initial condition $\mu_i^\circ$, denote $h_i^t:=g_{s}\ast \mu_i^t$, and assume $\mu_1,\mu_2\in \Ld^\infty([0,T];\Ld^\infty(\R^d))$ and $\nabla^2h_2\in\Ld^1([0,T];\Ld^\infty(\R^d))$ for some $T>0$. Then, for all $t\in[0,T]$,
\begin{align}\label{eq:proofsmooth20}
&\int_{\R^d}\int_{\R^d} g_{s}(x-y)d(\mu_1^t-\mu_2^t)(x)d(\mu_1^t-\mu_2^t)(y)\\
\le~& e^{C\int_0^t\|\nabla^2h_2^u\|_{\Ld^\infty}du}\int_{\R^d}\int_{\R^d} g_{s}(x-y)d(\mu_1^\circ-\mu_2^\circ)(x)d(\mu_1^\circ-\mu_2^\circ)(y).\nonumber
\end{align}
\end{lem}

\begin{proof}
By Lemma~\ref{lem:proofsmooth}, we only need to consider the case $d-2<s<d$, $s\ge0$. Proposition~\ref{prop:existence}(ii) below yields $\nabla (h_1 -h_2)\in\Ld^\infty(\R^+;\Ld^2(\R^d\times\R,|\xi|^\gamma dxd\xi))$. Combining this with the boundedness assumptions, all integration by parts arguments in the sequel may be justified.

Just as in~\eqref{eq:proofsmooth}, the time derivative of the left-hand side of~\eqref{eq:proofsmooth20} is
\begin{align}
&\partial_t\int_{\R^d}\int_{\R^d} g_{s}(x-y)d(\mu_1^t-\mu_2^t)(x)d(\mu_1^t-\mu_2^t)(y)\label{eq:proofsmooth2}\\
=~&-2\int_{\R^d} |\nabla (h_1^t-h_2^t)|^2\mu_1^t-2\int_{\R^d} \nabla h_2^t\cdot \nabla (h_1^t-h_2^t)\, (\mu_1^t- \mu_2^t).\nonumber
\end{align}
The first term in the right-hand side is nonpositive, so it suffices to estimate the second one. Using the relations $-\Div(|\xi|^\gamma\nabla h_i^t)=\mu_i^t\delta_{\R^d\times\{0\}}$, for $i=1,2$, we find the following proxy for the Delort-type formula~\eqref{eq:delortform}: for all $1\le k\le d$,
\begin{align*}
-2\partial_k (h_1^t-h_2^t)\, (\mu_1^t\delta_{\R^d\times\{0\}}- \mu_2^t\delta_{\R^d\times\{0\}})&=2\partial_k (h_1^t-h_2^t)\,\Div(|\xi|^\gamma\nabla (h_1^t-h_2^t))\\
&=\sum_{l=1}^{d+1}\partial_l\left(2|\xi|^\gamma\partial_k(h_1^t-h_2^t)\partial_l(h_1^t-h_2^t)-\delta_{kl}|\xi|^\gamma|\nabla (h_1^t-h_2^t)|^2\right).
\end{align*}
Combining this with an integration by parts, we obtain
\begin{align*}
2\int_{\R^d} \nabla h_2^t\cdot \nabla (h_1^t-h_2^t)\, (\mu_1^t- \mu_2^t)&=-\sum_{k=1}^d\sum_{l=1}^{d+1}\int_{\R^d\times\R} \partial_l\left(2|\xi|^\gamma\partial_k(h_1^t-h_2^t)\partial_l(h_1^t-h_2^t)-\delta_{kl}|\xi|^\gamma|\nabla (h_1^t-h_2^t)|^2\right)\partial_k h_2^t\\
&=\sum_{k=1}^d\sum_{l=1}^{d+1}\int_{\R^d\times\R} |\xi|^\gamma\left(2\partial_k(h_1^t-h_2^t)\partial_l(h_1^t-h_2^t)-\delta_{kl}|\nabla (h_1^t-h_2^t)|^2\right)\partial_{kl} h_2^t.
\end{align*}
Hence, arguing as in Lemma~\ref{lem:proofsmooth}, an integration by parts yields
\begin{align}
\left|\int_{\R^d} \nabla h_2^t\cdot \nabla (h_1^t-h_2^t)\, (\mu_1^t- \mu_2^t)\right|&\lesssim\|\nabla^2h_2^t\|_{\Ld^\infty}\int_{\R^d\times\R} |\xi|^\gamma|\nabla(h_1^t-h_2^t)|^2\nonumber\\
&=\|\nabla^2h_2^t\|_{\Ld^\infty}\int_{\R^d\times\R} (h_1^t-h_2^t)(\mu_1^t\delta_{\R^d\times\{0\}}-\mu_2^t\delta_{\R^d\times\{0\}})\nonumber\\
&=\|\nabla^2h_2^t\|_{\Ld^\infty}\int_{\R^d} (h_1^t-h_2^t)(\mu_1^t-\mu_2^t)\nonumber\\
&=\|\nabla^2h_2^t\|_{\Ld^\infty}\int_{\R^d}\int_{\R^d} g_s(x-y)d(\mu_1^t-\mu_2^t)(x)d(\mu_1^t-\mu_2^t)(y),\label{eq:ippRiesz}
\end{align}
and the result~\eqref{eq:proofsmooth20} follows.
\end{proof}

\subsection{Properties of the fractional porous medium equation}\label{chap:propeqn}

Let us first state for weak solutions to~\eqref{eq:pdelim} an existence result as well as some basic properties. We refer to~\cite{Caffarelli-Vazquez-11,Caffarelli-Soria-Vasquez-13} for $d-2<s<d$, $s\ge0$, and to~\cite{Lin-Zhang-00,Ambrosio-Serfaty-08,Serfaty-Vasquez-14,Bertozzi-Laurent-Leger-12} for $s=d-2$, $d\ge2$. For $0\le s<d-2$, we refer to~\cite[Section~4]{Carrillo-Choi-Hauray-14}, but there existence is proven only for short times, while in the present repulsive context it can easily be extended to all times, using that no blow-up can occur in finite time.\footnote{This follows from the observation that for $0\le s< d-2$ we have $\Delta g_s(x)=-s(d-2-s)c_{d,s}^{-1}|x|^{-s-2}$, and hence for all $p\ge1$ we find (formally) along solutions
\[\partial_t\int(\mu^t)^p=(p-1)\int (\mu^t)^{p} \Delta h^t\le0.\]}

\begin{prop}[Existence for the fractional porous medium equation]\label{prop:existence}
Let $0\le s<d$.
\begin{enumerate}[(i)]
\item Existence: Let $\mu^\circ\in\Pc(\R^d)\cap\Ld^\infty(\R^d)$, and in the case $d-2<s<d$, $s\ge0$ also assume that $|\mu^\circ(x)|\le Ae^{-a|x|}$ for some $a,A>0$. Then, there exists a (global) weak solution $\mu^t$ to~\eqref{eq:pdelim} in $\Ld^\infty(\R^+;\Pc(\R^d)\cap\Ld^\infty(\R^d))$, which is unique in this class in the case $0\le s\le d-2$, $d\ge2$.
\item General properties: Let $\mu^\circ\in\Pc(\R^d)\cap\Ld^\infty(\R^d)$, and in the case $s=0$ also assume $\int_{\R^d}\log(2+|x|)|\mu^\circ(x)|<\infty$. Any weak solution $\mu^t$ to~\eqref{eq:pdelim} on $\R^+\times\R^d$ with initial condition $\mu^\circ$ then satisfies
\[\int_{\R^d}\int_{\R^d} g_s(x-y)d\mu^t(x)d\mu^t(y)\le \int_{\R^d}\int_{\R^d} g_s(x-y)d\mu^\circ(x)d\mu^\circ(y),\]
where the left-hand side remains finite. Moreover, for all $t\ge0$,
\[\int_{\R^d}\int_{\R^d} g_s(x-y)d(\mu^t-\mu^\circ)(x)d(\mu^t-\mu^\circ)(y)=\begin{cases}\int_{\R^d} |\nabla (h^t-h^\circ)|^2,&\text{if $s=d-2$, $d\ge2$};\\
\int_{\R^d\times\R} |\xi|^\gamma|\nabla (h^t-h^\circ)|^2,&\text{if $d-2<s<d$, $s\ge0$;}\end{cases}\]
where both sides remain finite.
Also, if $\mu^\circ$ is compactly supported, then $\mu^t$ remains compactly supported for all $t\ge0$.
\end{enumerate}
\end{prop}

In the case $d-2<s<d$, $s\ge0$, uniqueness remains an open problem:
it has been obtained in dimension~$1$ by~\cite{Biler-Karch-Monneau-10} (integrating the equation with respect to $x$ and then considering viscosity solutions), but in higher dimensions no result is known (cf. \cite{Caffarelli-Vazquez-11,Caffarelli-Soria-Vasquez-13}).
Nevertheless, as a consequence of the stability result of Lemma~\ref{lem:proofsmooth2}, we easily find that uniqueness of bounded weak solutions always follows from the existence of a smoother solution, so the problem is somehow reduced to a regularity question:

\begin{cor}[Weak-strong uniqueness for the fractional porous medium equation]\label{cor:uniqueness}
Let $0\vee(d-2)\le s<d$. Let $\mu^\circ\in\Pc(\R^d)\cap\Ld^\infty(\R^d)$, and in the case $s=0$ also assume $\int_{\R^d}\log(2+|x|)|\mu^\circ(x)|dx<\infty$. Assume that equation~\eqref{eq:pdelim} admits a weak solution $\mu^t$ with initial condition $\mu^\circ$, such that $\mu,\nabla^2h\in\Ld^\infty([0,T];\Ld^\infty(\R^d))$ for some $T>0$. Then, $\mu$ is the unique weak solution to~\eqref{eq:pdelim} up to time $T$ in the class $\Ld^\infty([0,T];\Ld^\infty(\R^d))$.
\end{cor}

\begin{proof}
Let $\mu^t$ be a weak solution to~\eqref{eq:pdelim} as in the statement, and let $\nu^t$ denote another weak solution to~\eqref{eq:pdelim} in $\Ld^\infty([0,T];\Ld^\infty(\R^d))$. By Lemma~\ref{lem:proofsmooth2}, we may then conclude
\begin{align}\label{eq:uniquenessprinter}
\int_{\R^d}\int_{\R^d} g_s(x-y)d(\mu^t-\nu^t)(x)d(\mu^t-\nu^t)(y)\le0,
\end{align}
for all $t\in[0,T]$. For $d-2<s<d$, $s\ge0$, Proposition~\ref{prop:existence}(ii)
gives $\nabla g_s\ast(\mu^t-\nu^t)\in\Ld^2(\R^d,|\xi|^\gamma dxd\xi)$, so that~\eqref{eq:uniquenessprinter} becomes by integration by parts
\[\int_{\R^d\times\R} |\xi|^\gamma |\nabla g_s\ast (\mu^t-\nu^t)|^2\le0.\]
This proves $\nabla g_s\ast \mu^t=\nabla g_s\ast\nu^t$, and hence, applying the operator $-\Div(|\xi|^\gamma\cdot)$ to both sides, $\mu^t=\nu^t$ for all $t\in[0,T]$. We may argue similarly in the Coulomb case $s=d-2$, $d\ge2$.
\end{proof}

As the following lemma shows, the required boundedness of $\nabla^2h^t$ is implied by a sufficient amount of Hölder regularity for $\mu^t$.

\begin{lem}\label{lem:hbounded}
Let $0\vee(d-2)\le s<d$. Let $\mu\in\Pc(\R^d)\cap C^\sigma(\R^d)$ for some $\sigma>2-d+s$, and denote by $h^\mu:=g_s\ast \mu$ the associated Riesz potential~\eqref{eq:defriesz} on $\R^d$.
If $s=d-1$, we further assume $\nabla\mu\in \Ld^{p_0}(\R^d)$ for some $p_0<\infty$. Then, we have
\begin{align}
\|(\nabla h^\mu,\nabla^2h^\mu)\|_{\Ld^\infty}&\lesssim \|\mu\|_{\Ld^1}+\|\mu\|_{C^\sigma}.\label{eq:boundht-1}
\end{align}
Moreover, if $s=d-2\ge0$ we have $\|\nabla^2 h^\mu\|_{\Ld^p}\lesssim_p\|\mu\|_{\Ld^p}$ for all $1<p<\infty$, and if $s=d-1$ we have $\|\nabla^2 h^\mu\|_{\Ld^p}\lesssim_p\|\nabla\mu\|_{\Ld^p}$ for all $p_0\le p<\infty$, $p>1$.
\end{lem}

\begin{proof}
Without loss of generality we may assume $\mu\in C^\infty_c(\R^d)$, as the claimed result then follows by an obvious approximation argument.
We first prove that for any $\mu\in C^\infty_c(\R^d)$ the Riesz potential $h^\mu$ satisfies~\eqref{eq:boundht-1}.
We only argue for the second gradient $\|\nabla^2h^\mu\|_{\Ld^\infty}$, the other part being similar and easier.
Let $\chi\in C^\infty_c(\R^d)$ be symmetric around $0$, with $\chi=1$ in $B_1$ and $\chi=0$ outside $B_2$. If $d-2\le s<d-1$, decomposing
\begin{align*}
\nabla^2h^\mu(x)&=\int_{\R^d} g_s(x-y)\nabla^2\mu(y)dy\\
&=\int_{\R^d} g_s(x-y)(1-\chi(x-y))\nabla^2\mu(y)dy+\int_{\R^d} g_s(x-y)\chi(x-y)\nabla^2_y(\mu(y)-\mu(x))dy,
\end{align*}
we find by multiple integrations by parts
\begin{align*}
\nabla^2h^\mu(x)&=\int_{\R^d}\nabla^2g_s(x-y)(1-\chi(x-y))\mu(y)dy+\int_{\R^d}\nabla^2 g_s(x-y)\chi(x-y)(\mu(y)-\mu(x))dy\\
&\qquad-\mu(x)\int_{\R^d} g_s(x-y)\nabla^2\chi(x-y)dy,
\end{align*}
and hence, for any $2-d+s<\sigma<1$,
\begin{align*}
|\nabla^2h^\mu(x)|&\lesssim \int_{|x-y|\ge1}|x-y|^{-s-2}|\mu(y)|dy+\|\mu\|_{C^\sigma}\int_{|x-y|\le2}|x-y|^{\sigma-s-2}dy+\|\mu\|_{\Ld^\infty}\int_{|x-y|\le2}|x-y|^{-s}dy\\
&\lesssim \|\mu\|_{\Ld^1}+\|\mu\|_{C^\sigma},
\end{align*}
that is~\eqref{eq:boundht-1}. If $d-1\le s<d$, rather decomposing
\begin{align*}
\nabla^2h^\mu(x)
&=\frac12\int_{\R^d} \nabla g_s(x-y)\otimes(\nabla\mu(y)-\nabla\mu(2x-y))dy\\
&=\int_{\R^d} \nabla^2 g_s(x-y)(1-\chi(x-y))\mu(y)dy-\int_{\R^d} \nabla g_s(x-y)\otimes\nabla\chi(x-y)\mu(y)dy\\
&\qquad+\frac12\int_{\R^d} \nabla g_s(x-y)\otimes(\nabla\mu(y)-\nabla\mu(2x-y))\chi(x-y)dy,
\end{align*}
the result~\eqref{eq:boundht-1} again follows from a direct computation. As far as the additional $\Ld^p$-boundedness is concerned, it is a direct consequence of the $\Ld^p$-boundedness of Riesz transforms for $1<p<\infty$, simply noting that we have $\nabla^2 h^\mu\simeq \nabla^2(-\Delta)^{-1}\mu$ for $s=d-2\ge0$, and $\nabla^2h^\mu\simeq \nabla(-\Delta)^{-1/2}\nabla\mu$ for $s=d-1$.
\end{proof}

Motivated by these considerations, we would like to prove at least that the regularity of the initial condition is conserved along the flow, so that the boundedness of $\nabla^2h^t$ would be ensured by the above lemma for sufficiently smooth initial data.
In the Coulomb case $s=d-2$, $d\ge2$, any weak solution $\mu^t$ to~\eqref{eq:pdelim} can be shown to belong to $\Ld^\infty(\R^+;\Pc(\R^d)\cap C^\sigma(\R^d))$ whenever $\mu^\circ\in \Pc(\R^d)\cap C^\sigma(\R^d)$ with non-integer $\sigma>0$ (cf.~\cite[Theorem~1]{Lin-Zhang-00}, which is easily rewritten in any dimension). By a similar but easier argument, the corresponding result can also be proven to hold in the case $0\le s<d-2$, $d\ge3$.
In the case $d-2<s\le d-1$, $s\ge0$, a recent result~\cite{Zhou-16} 
shows that there exists a time $T>0$ (depending on initial data) such that any weak solution $\mu^t$ to~\eqref{eq:pdelim} belongs to $\Ld^\infty([0,T];H^\sigma(\R^d))$, whenever $\mu^\circ\in\Pc(\R^d)\cap H^\sigma(\R^d)$ with $\sigma>\frac d2+1$.
In the case $d-1<s<d$, however, even such a short-time regularity result is unknown.
This is why we needed to add some regularity assumption in the statement of Theorem~\ref{th:mfl}.

\subsection{Modulated energy and elementary properties}\label{chap:modenergy}
Let $0\le s<d$, and let $\mu^\circ,\mu_N^\circ,\mu^t,\mu_N^t$ be as in the statement of Theorem~\ref{th:mfl}, for some $T\in(0,\infty)$. Let $N\ge1$. Since $\mu_N^\circ$ is assumed to have bounded energy, and since the energy is decreasing along the flow, we find
\[\sup_{t\in[0,T]}\frac1{N^2}H_N(x_{1,N}^t,\ldots,x_{N,N}^t)=\sup_{t\in[0,T]}\frac1{N^2}\sum_{i\ne j}^Ng_s(x_{i,N}^t-x_{j,N}^t)\le \frac1{N^2}\sum_{i\ne j}^Ng_s(x_{i,N}^\circ-x_{j,N}^\circ)<\infty.\]
If $0<s<d$, since $g_s$ is nonnegative, this proves
\begin{align}\label{eq:distancepoints}
\eta_N:=\min_{i\ne j}^N\inf_{t\in[0,T]}|x_{i,N}^t-x_{j,N}^t|>0.
\end{align}
If $s=0$, $g_0$ changes sign and some more work is then needed: noting that by symmetry
\[\partial_t\frac1N\sum_{i=1}^Nx_{i,N}^t=-\frac1{N^2}\sum_{i\ne j}^N\nabla g_0(x_{i,N}^t-x_{j,N}^t)=\frac1{N^2}\sum_{i\ne j}^N\frac{x_{i,N}^t-x_{j,N}^t}{|x_{i,N}^t-x_{j,N}^t|^2}=0,\]
a direct computation yields
\begin{align*}
\partial_t\frac1{N^2}\sum_{i\ne j}^N|x_{i,N}^t-x_{j,N}^t|^2&=\partial_t\frac2N\sum_{i=1}^N|x_{i,N}^t|^2-\partial_t\frac2{N^2}\bigg|\sum_{i=1}^Nx_{i,N}^t\bigg|^2\\
&=\frac4{N^2}\sum_{i\ne j}^Nx_{i,N}^t\cdot\frac{x_{i,N}^t-x_{j,N}^t}{|x_{i,N}^t-x_{j,N}^t|^2}\\
&=\frac2{N^2}\sum_{i\ne j}^N(x_{i,N}^t-x_{j,N}^t)\cdot\frac{x_{i,N}^t-x_{j,N}^t}{|x_{i,N}^t-x_{j,N}^t|^2}=\frac{2(N-1)}N,
\end{align*}
and hence
\begin{align*}
&\sup_{t\in[0,T]}\frac1{N^2}\sum_{i\ne j}^N(g_0(x_{i,N}^t-x_{j,N}^t)+c_{d,0}^{-1}|x_{i,N}^t-x_{j,N}^t|^2)\\
\le~& \frac1{N^2}\sum_{i\ne j}^Ng_0(x_{i,N}^\circ-x_{j,N}^\circ)+\frac{c_{d,0}^{-1}}{N^2}\sum_{i\ne j}^N|x_{i,N}^\circ-x_{j,N}^\circ|^2+Tc_{d,0}^{-1}\frac{2(N-1)}N<\infty.
\end{align*}
As $g_0(u)+c_{d,0}^{-1}u^2\ge0$ for all $u$, this proves that~\eqref{eq:distancepoints} also holds in the case $s=0$.

Next, we recall the truncation procedure introduced in~\cite{Petrache-Serfaty-14}, which serves to make energies finite without removing the diagonal.
For fixed $N\ge1$, let $\eta>0$ be small enough such that $2\eta<1\wedge\eta_N$. Then define
\[\mu_{N,\eta}^t:=\frac1N\sum_{i=1}^N\delta^{(\eta)}_{x_{i,N}^t}\in\Pc(\R^d),\]
where $\delta^{(\eta)}_z$ denotes the uniform unit Dirac mass on the sphere $\partial B(z,\eta)$.
Denote for simplicity
\[h^t:=g_s\ast\mu^t,\qquad h_N^t:=g_s\ast\mu_N^t,\qquad h_{N,\eta}^t:=g_s\ast\mu^t_{N,\eta},\]
and use the same notation for their extensions to $\R^d\times\R$ as in Section~\ref{chap:caffsilv}. Also define $g_{s,\eta}:= g_s(\eta)\wedge g_s$.
Noting that by symmetry
\[-\Div(|\xi|^\gamma\nabla g_{s,\eta})=\delta_0^{(\eta)}\delta_{\R^d\times\{0\}},\]
where $\delta_0^{(\eta)}\delta_{\R^d\times\{0\}}$ denotes the unit Dirac mass on $\partial B_\eta\times\{0\}$, we find
\[h_{N,\eta}^t(x,\xi)=\frac1N\sum_{i=1}^Ng_{s,\eta}(x-x_{i,N}^t,\xi).\]

Let us now introduce our notation for the small balls around the particle locations, which we will be crucially using in the proof: for all $R>0$, let $\B_N^t(R)$ denote a union of disjoint balls
\begin{align}\label{def:smallballs}
\B_N^t(R):=\bigcup_{m=1}^{M_N^t(R)}B(y_{m,N}^t,r_{m,N}^t),
\end{align}
with total radius $R=\sum_mr_{m,N}^t$ and such that $x_{i,N}^t\in\B_N^t(R)$ for all $1\le i\le N$. These balls will be carefully chosen in Section~\ref{chap:ball0} below.

As already announced, for all $N\ge1$, we will consider the following {\it modulated energy}
\begin{align}\label{eq:modnrj}
\Ec_N(t):=\iint_{D^c}g_s(x-y)d(\mu_N^t-\mu^t)(x)d(\mu_N^t-\mu^t)(y),
\end{align}
where $D:=\{(x,x):x\in\R^d\}$ denotes the diagonal. This quantity can be thought of as a natural renormalization of the $H^{-(d-s)/2}$-distance in the presence of Dirac masses.
Its main property is as follows:

\begin{lem}[Modulated energy]\label{lem:propmodenergy}
For all $t\ge0$, if the sequence $(\mu_N^t)_N$ is tight, then the following two conditions are equivalent:
\begin{enumerate}[(i)]
\item $\limsup_{N\uparrow\infty}\Ec_N(t)\le0$;
\item $\mu_N^t\cvf*\mu^t$ and $\iint_{D^c}g_s(x-y)d\mu_N^t(x)d\mu_N^t(y)\to\int_{\R^d}\int_{\R^d} g_s(x-y)d\mu^t(x)d\mu^t(y)$.
\end{enumerate}
\end{lem}

\begin{proof}
Property~(ii) clearly implies~(i) (and even $\Ec_N(t)\to0$), so it suffices to check the converse. Assume that $\limsup_{N}\Ec_N(t)\le0$. By tightness, up to extraction of a subsequence, the Prokhorov theorem gives $\mu_N^t\cvf*\nu^t$ for some $\nu^t\in\Pc(\R^d)$. For any $K>0$, we may write
\begin{align*}
\iint_{D^c}g_s(x-y)d\mu^t_N(x)d\mu_N^t(y)&\ge\iint_{D^c}K\wedge g_s(x-y)d\mu^t_N(x)d\mu_N^t(y)\\
&=-\frac KN+\int_{\R^d}\int_{\R^d} K\wedge g_s(x-y)d\mu^t_N(x)d\mu_N^t(y),
\end{align*}
and hence, successively passing to the limits $N\uparrow\infty$ and $K\uparrow\infty$, we find
\begin{align}\label{eq:ineqmuNconvnunrj}
\liminf_{N\uparrow\infty} \iint_{D^c}g_s(x-y)d\mu^t_N(x)d\mu_N^t(y)\ge \int_{\R^d}\int_{\R^d} g_s(x-y)d\nu^t(x)d\nu^t(y).
\end{align}
Combining this with convergence $\mu_N^t\cvf*\nu^t$ and with assumption $\limsup_{N}\Ec_N(t)\le0$, we obtain
\begin{align*}
0&\ge \limsup_{N\uparrow\infty}\iint_{D^c}g_s(x-y)d\mu_N^t(x)d\mu_N^t(y)-2\int_{\R^d}\int_{\R^d} g_s(x-y)d\nu^t(x)d\mu^t(y)+\int_{\R^d}\int_{\R^d} g_s(x-y)d\mu^t(x)d\mu^t(y)\\
&\ge\int_{\R^d}\int_{\R^d} g_s(x-y)d\nu^t(x)d\nu^t(y)-2\int_{\R^d}\int_{\R^d} g_s(x-y)d\nu^t(x)d\mu^t(y)+\int_{\R^d}\int_{\R^d} g_s(x-y)d\mu^t(x)d\mu^t(y)\\
&=\int_{\R^d}\int_{\R^d} g_s(x-y)d(\nu^t-\mu^t)(x)d(\nu^t-\mu^t)(y).
\end{align*}
The result then follows, noting that $\mu^t$ has bounded energy by Proposition~\ref{prop:existence}, that $\nu^t$ has bounded energy by~\eqref{eq:ineqmuNconvnunrj}, and noting that for any two Radon measures $\mu,\nu$ with finite energy we have
\[\int_{\R^d}\int_{\R^d} g_s(x-y)d(\nu-\mu)(x)d(\nu-\mu)(y)\ge0,\]
with equality if only if $\mu=\nu$ (see e.g.~\cite[Theorem~9.8]{Lieb-Loss} for $0<s<d$, and~\cite[Lemma~1.8]{Saff-Totik} for $s=0$).
\end{proof}

In the case of bounded weak solutions $\mu_1^t,\mu_2^t$ to~\eqref{eq:pdelim} as given by Proposition~\ref{prop:existence}, the following identity follows from an integration by parts and was crucially used in the proofs of Lemmas~\ref{lem:proofsmooth} and~\ref{lem:proofsmooth2} (cf.~\eqref{eq:ippCoul} and~\eqref{eq:ippRiesz}):
\begin{align}\label{eq:ippformmodenergy}
\int_{\R^d}\int_{\R^d} g_s(x-y)d(\mu_1^t-\mu_2^t)(x)d(\mu_1^t-\mu_2^t)(y)=\begin{cases}\int_{\R^d}|\nabla(h_1^t-h_2^t)|^2,&\text{if $s=d-2$, $d\ge2$};\\
\int_{\R^d\times\R}|\xi|^\gamma|\nabla(h_1^t-h_2^t)|^2,&\text{if $d-2<s<d$, $s\ge0$.}\end{cases}
\end{align}
Now we would need a corresponding identity in the context of the modulated energy $\Ec_N(t)$. Since $\nabla h_N^t$ does not belong to $\Ld^2(\R^d)$ or $\Ld^2(\R^d\times\R,|\xi|^\gamma dxd\xi)$, a regularization is then needed. Besides the modulated energy $\Ec_N$, we thus define the following $\eta$-approximation, based on the truncation introduced above:
\[\Ec_{N,\eta}(t):=\begin{cases}\int_{\R^d}|\nabla(h_{N,\eta}^t-h^t)|^2,&\text{if $s=d-2$, $d\ge2$};\\\int_{\R^d\times\R}|\xi|^\gamma|\nabla(h_{N,\eta}^t-h^t)|^2,&\text{if $d-2<s<d$, $s\ge0$}.\end{cases}\]
An integration by parts then yields the following proxy for identity~\eqref{eq:ippformmodenergy}, showing that the difference between the modulated energy $\Ec_N(t)$ and its approximation $\Ec_{N,\eta}(t)$ just comes from the diagonal terms (which are indeed excluded in $\Ec_N(t)$ but not in $\Ec_{N,\eta}(t)$). We refer to~\cite[Section~2.1]{Petrache-Serfaty-14} for a detailed proof.
\begin{lem}[Approximate modulated energy]\label{lem:linkENeta0}
Let $0\vee(d-2)\le s<d$. For all $t\ge0$, $N\ge1$ and $\eta>0$,
\begin{align*}
\Ec_{N,\eta}(t)=\Ec_N(t)+\frac{g_s(\eta)}N+o_{\eta}^{(N)}(1),
\end{align*}
where for any fixed $N$ we have $o_{\eta}^{(N)}(1)\to0$ as $\eta\downarrow0$.
\end{lem}

\subsection{Gronwall argument on the modulated energy}
By Lemma~\ref{lem:propmodenergy}, in order to prove convergence $\mu_N^t\cvf*\mu^t$ as well as convergence of energies, up to tightness issues, it suffices to check that $\limsup_N\Ec_N(t)\le0$. This is achieved by a Gronwall argument. From now on we focus on the Riesz case $d-2<s<d$, $s\ge0$. The Coulomb case $s=d-2$, $d\ge2$ can be treated in exactly the same way, but is actually easier since it does not require to use the extension representation of Section~\ref{chap:caffsilv}.

\begin{prop}\label{prop:MFL}
Let $d-2<s<d$, $s\ge0$. Let $\mu^\circ$ be a probability measure such that equation~\eqref{eq:pdelim} admits a solution $\mu^t$ that belongs to $\Ld^\infty([0,T];\Pc(\R^d)\cap C^{\sigma}(\R^d))$ for some $T>0$ and some $\sigma>2-d+s$. In the logarithmic case $s=0$, $d=1$, also assume that $\int_{\R}\log(2+|x|)|\mu^\circ(x)|dx<\infty$, and $\nabla\mu\in \Ld^\infty([0,T];\Ld^p(\R))$ for some $p<\infty$. Let $\mu_N^\circ\cvf*\mu^\circ$, assume
\begin{align}\label{eq:asnrjbounded}
\limsup_{N\uparrow\infty}\iint_{x\ne y}g_s(x-y)d\mu_N^\circ(x)d\mu_N^\circ(y)<\infty,
\end{align}
and let $\mu_N^t$ be defined by~\eqref{eq:vortexdyn}--\eqref{eq:empiricalmeas}. Assume that, for all $t\in[0,T]$, the collection $\B_N^t(R_N^t)$ can be chosen with $R_N^t\to0$ in such a way that
\begin{align}\label{eq:cond1}
\liminf_{N\uparrow\infty}\liminf_{\eta\downarrow0}\bigg(\int_{\B_N^t(R_N^t)\times\R}|\xi|^\gamma|\nabla h_{N,\eta}^t|^2-\frac{g_s(\eta)}N\bigg)\ge 0,
\end{align}
and, denoting $g_s^+(t):=c_{d,s}^{-1}t^{-s}$ for $s>0$ and $g_0^+(t):=c_{d,0}^{-1}(-\log t)\vee0$ otherwise,
\begin{align}\label{eq:cond2}
\lim_{N\uparrow\infty}\frac1{N^2}\sum_{i=1}^Ng_s^+(d(x_{i,N}^t,\partial\B_N^t(R_N^t)))=0.
\end{align}
Then, for all $t\in[0,T]$, we have $\Ec_N(t)\le C_t(\Ec_N(0)+o_N(1))$.
\end{prop}

\begin{rem}\label{rem:conditions}
In the ideal case when all particles remain well-separated, that is with a minimal distance $\eta_N$ of order $N^{-1/d}$, then, taking $\B_N^t(R_N^t)$ to be the union of balls of radius $R_N^t/N$ centered at the points $x_{i,N}^t$'s, with $R_N^t/N\ll N^{-1/d}$, condition~\eqref{eq:cond2} simply becomes $g_s(R_N^t/N)/N\ll1$. On the other hand, neglecting interactions between particles, hence focusing on the (divergent) self-interactions, we formally find
\begin{align*}
\int_{\B_N^t(R_N^t)\times\R}|\xi|^\gamma|\nabla h_{N,\eta}^t|^2&=\frac1{N^2}\sum_{i=1}^N\int_{|x-x_{i,N}^t|<R_N^t/N}|\xi|^\gamma|\nabla g_{s,\eta}(x-x_{i,N}^t,\xi)|^2+\ldots\\
&=\frac1{N}\int_{\eta<|x|<R_N^t/N}|\xi|^\gamma|\nabla g_{s}(x,\xi)|^2+\ldots\\
&=\frac1{N}(g_s(\eta)-g_s(R_N^t/N))+\ldots,
\end{align*}
so that condition~\eqref{eq:cond1} would amount to requiring $g_s(R_N^t/N)/N\ll1$, which is thus just the same as condition~\eqref{eq:cond2}. In other words, for $s>0$, both conditions would then take the form $R_N^t\gg N^{-(1-s)/s}$, which is compatible with $R_N^t\to0$ only if $s<1$. In Section~\ref{chap:ball0}, we prove that a consistent choice of the small balls $\B_N^t(R_N^t)$ is indeed possible whenever $0\le s<1$.

To go beyond the restriction $s<1$ via this approach, we would need to modify Proposition~\ref{prop:MFL}, in particular by refining the (blind) approximation argument used in Step~2 of the proof below, in order to relax the smallness condition for the total radius $R_N^t\to0$. To do that, precise microscopic information on the particle dynamics would become needed. Getting a handle on such information seems however to be a difficult task and is not pursued here.
\end{rem}

\begin{proof}
By the regularity assumption for $\mu^t$, Lemma~\ref{lem:hbounded} ensures that we have $\|(\nabla_xh^t,\nabla_x^2h^t)\|_{\Ld^\infty}\lesssim_t1$, and also, in the case $s=0$, $d=1$, $\|\nabla_x^2h^t\|_{\Ld^p}\lesssim_t1$ for some $p<\infty$.
We split the proof into four steps.

\medskip
\noindent{\it Step~1: Time-derivative of $\Ec_N(t)$ and modulated stress-energy tensor.}
In this step, we prove equality
\begin{align}\label{eq:derENt0r}
\partial_t\Ec_N(t)&=-\int_{\R^d\times\R}|\xi|^\gamma\nabla^2_xh^t(x): T_N^t(x,\xi)dxd\xi\\
&\qquad-2\int_{\R^d}\bigg|\pv\int_{\R^d\setminus\{x\}}\nabla g_s(x-y)d(\mu_N^t-\mu^t)(y)\bigg|^2d\mu_N^t(x),\nonumber
\end{align}
where we use the usual principal value symbol
\[\pv\int_{\R^d\setminus\{x\}}:=\lim_{r\downarrow0}\int_{\R^d\setminus B(x,r)}\]
and where the modulated stress-energy tensor $T_N^t=(T_N^{t;kl})_{k,l=1}^{d+1}$ is defined as follows: for all $1\le k,l\le d$,
\begin{align}\label{eq:TNt}
T_N^{t;kl}(x,\xi):=&~2\iint_{D^c}\partial_k g_s(x-y,\xi)\partial_l g_s(x-z,\xi)d(\mu_N-\mu)(y)d(\mu_N-\mu)(z)\\
&\qquad-\delta_{kl}\iint_{D^c}\nabla g_s(x-y,\xi)\cdot\nabla g_s(x-z,\xi)d(\mu_N-\mu)(y)d(\mu_N-\mu)(z).\nonumber
\end{align}
Moreover, as checked at the end of this step, the integrals in~\eqref{eq:derENt0r} are summable: more precisely, we prove that $|T_N^t|$ belongs to $\Ld^1(\R^d\times\R,|\xi|^\gamma dxd\xi)$ if $s>0$, and that $|\nabla^2_xh^t(x)||T_N^t(x,\xi)|$ belongs to $\Ld^1(\R^d\times\R,|\xi|^\gamma dxd\xi)$ if $s=0$, $d=1$.
Although the second term in the right-hand side of~\eqref{eq:derENt0r} is nonpositive, we do not bound it by~$0$ yet, contrarily to what is done in the proof of Lemmas~\ref{lem:proofsmooth} and~\ref{lem:proofsmooth2}, since it will be useful in Step~2 below to absorb some error terms.

Using the equations satisfied by $\mu^t$
and by the trajectories $x_{i,N}^t$, and noting that the gradient $\nabla h^t$ is given by
\[\nabla h^t(x)=\pv\int_{\R^d\setminus\{x\}}\nabla g_s(x-y)d\mu^t(y),\]
where the principal value may only be omitted for $s<d-1$, we find the following expression for the time-derivative of the modulated energy $\Ec_N(t)$ defined in~\eqref{eq:modnrj}:
\begin{align*}
\partial_t\Ec_N(t)&=\partial_t\int_{\R^d}\int_{\R^d} g_s(x-y)d\mu^t(x)d\mu^t(y)+\partial_t\frac1{N^2}\sum_{i\ne j}^Ng_s(x_{i,N}^t-x_{j,N}^t)-\partial_t\frac2N\sum_{i=1}^N\int_{\R^d} g_s(x_{i,N}^t-y)d\mu^t(y)\\
&=-2\int_{\R^d} \nabla h^t(x)\cdot \pv\int_{\R^d\setminus\{x\}}\nabla g_s(x-y)d\mu^t(y)d\mu^t(x)-\frac2{N}\sum_{i=1}^N\bigg|\frac1N\sum_{j,j\ne i}\nabla g_s(x_{i,N}^t-x_{j,N}^t)\bigg|^2\\
&\qquad+\frac2{N^2}\sum_{i\ne j}^N\nabla h^t(x_{i,N}^t)\cdot\nabla g_s(x_{i,N}^t-x_{j,N}^t)+\frac2N\sum_{i=1}^N\pv\int_{\R^d\setminus\{x_{i,N}^t\}} \nabla h^t(x)\cdot \nabla g_s(x-x_{i,N}^t)d\mu^t(x).
\end{align*}
Let us rearrange the terms as follows:
\begin{align*}
\partial_t\Ec_N(t)&=-2\int_{\R^d}\bigg|\pv\int_{\R^d\setminus\{x\}}\nabla g_s(x-y)d(\mu_N^t-\mu^t)(y)\bigg|^2d\mu_N^t(x)\\
&\qquad-2\int_{\R^d}\nabla h^t(x)\cdot\pv\int_{\R^d\setminus\{x\}}\nabla g_s(x-y)d\mu^t(y)d\mu^t(x)\\
&\qquad+2\int_{\R^d}\nabla h^t(x)\cdot\pv\int_{\R^d\setminus\{x\}}\nabla g_s(x-y)d\mu^t(y)d\mu_N^t(x)\\
&\qquad-2\int_{\R^d}\nabla h^t(x)\cdot\int_{\R^d\setminus\{x\}}\nabla g_s(x-y)d\mu_N^t(y)d\mu_N^t(x)\\
&\qquad+2\int_{\R^d}\pv\int_{\R^d\setminus\{y\}}\nabla h^t(x)\cdot \nabla g_s(x-y)\mu^t(x)\mu_N^t(y),
\end{align*}
and note that the last four terms in the right-hand side may be combined to yield the following simpler expression:
\begin{align}\label{eq:derENt-001}
\partial_t\Ec_N(t)&=-2\int_{\R^d}\bigg|\pv\int_{\R^d\setminus\{x\}}\nabla g_s(x-y)d(\mu_N^t-\mu^t)(y)\bigg|^2d\mu_N^t(x)\\
&\qquad-\underbrace{\iint_{D^c}(\nabla h^t(x)-\nabla h^t(y))\cdot \nabla g_s(x-y)d(\mu_N^t-\mu^t)(y)d(\mu_N^t-\mu^t)(x)}_{=:\,I_N(t)}.\nonumber
\end{align}
In the distributional sense on $\R^d$, using canonical regularizations, we may alternatively write
\begin{align}\label{eq:derENt-002}
I_N(t)=\langle S_N^t;\nabla h^t\rangle=\sum_{k=1}^d\langle S_N^{t;k};\partial_kh^t\rangle,
\end{align}
where $S_N^t=(S_N^{t;k})_{k=1}^d$ and, for all $1\le k\le d$,
\[S_N^{t;k}(x):=2(\mu_N^t-\mu^t)(x)\pv\int_{\R^d\setminus\{x\}}\partial_k g_s(x-y)d(\mu_N^t-\mu^t)(y).\]
Since $-\Div(|\xi|^\gamma\nabla g_s(x-x_0,\xi))=\delta_{x_0}(x)\delta_{\R^d\times\{0\}}(x,\xi)$ for all $x_0\in\R^d$, we have in the distributional sense on $\R^d\times\R$
\begin{align*}
S_N^{t,k}(x)\delta_{\R^d\times\{0\}}(x,\xi)&=-2\pv\iint_{D^c}\Div(|\xi|^\gamma\nabla g_s(x-z,\xi))\partial_k g_s(x-y,\xi)d(\mu_N^t-\mu^t)(z)d(\mu_N^t-\mu^t)(y)\\
&=-\pv\iint_{D^c}\big(\Div(|\xi|^\gamma\nabla g_s(x-z,\xi))\partial_k g_s(x-y,\xi)+\Div(|\xi|^\gamma\nabla g_s(x-y,\xi))\partial_k g_s(x-z,\xi)\big)\\
&\hspace{4cm}\times d(\mu_N^t-\mu^t)(z)d(\mu_N^t-\mu^t)(y).
\end{align*}
Now note the following algebraic identity in the distributional sense on $\R^d\times\R$: for all $1\le k\le d$,
\begin{align*}
&\Div(|\xi|^\gamma\nabla g_s(x-y,\xi))\partial_kg_s(x-z,\xi)+\Div(|\xi|^\gamma\nabla g_s(x-z,\xi))\partial_kg_s(x-y,\xi)\\
&\hspace{4cm}=\frac12\sum_{l=1}^{d+1}\big(\partial_l(|\xi|^\gamma G_s^{lk}(x,\xi;y,z))+\partial_l(|\xi|^\gamma G_s^{lk}(x,\xi;z,y))\big),
\end{align*}
where we have set
\begin{align}\label{eq:defGs}
G_s^{lk}(x,\xi;y,z):=2\partial_lg_s(x-y,\xi)\partial_kg_s(x-z,\xi)-\delta_{lk}\sum_{m=1}^{d+1}\partial_mg_s(x-y,\xi)\partial_mg_s(x-z,\xi).
\end{align}
This proves the (Delort-type) identity
\begin{align}\label{eq:SN-TN}
S_N^{t;k}(x)\delta_{\R^d\times\{0\}}(x,\xi)=-\sum_{l=1}^{d+1}\partial_l(|\xi|^\gamma T_N^{t;lk}(x,\xi))
\end{align}
for all $1\le k\le d$, and the conclusion~\eqref{eq:derENt0r} then follows from~\eqref{eq:derENt-001}, \eqref{eq:derENt-002} and an integration by parts.

We now turn to the claimed integrability of the modulated stress-energy tensor $T_N^t$.
We first consider the case $d-2<s<d$, $s>0$.
For that purpose, we begin with the bound
\begin{align}\label{eq:boundintegTN}
\int_{\R^d\times\R}|\xi|^\gamma|T_N^t|\lesssim~&\int_{\R^d\times\R}|\xi|^\gamma\iint_{D^c}|(x-y,\xi)|^{-s-1}|(x-z,\xi)|^{-s-1}d(\mu_N^t+\mu^t)(y)d(\mu_N^t+\mu^t)(z)dxd\xi\nonumber\\
=~&\iint_{D^c}\bigg(\int_{\R^d\times\R}|\xi|^\gamma|(x-y,\xi)|^{-s-1}|(x-z,\xi)|^{-s-1}dxd\xi\bigg)d(\mu_N^t+\mu^t)(y)d(\mu_N^t+\mu^t)(z).
\end{align}
Let us compute the integral over $\R^d\times\R$. Denoting for simplicity $c_{yz}:=(y+z)/2$ and $q:=s+1$, we decompose, for all $y\ne z$,
\[\int_{\R^d}(|x-y|^2+1)^{-q/2}(|x-z|^2+1)^{-q/2}dx=I_{yz}^1+I_{yz}^2+I_{yz}^3+I_{yz}^4,\]
where
\begin{align*}
I^1_{yz}:=&\int_{|x-y|\le\frac12|y-z|}(|x-y|^2+1)^{-q/2}(|x-z|^2+1)^{-q/2}dx,\\
I_{yz}^2:=&\int_{|x-z|\le\frac12|y-z|}(|x-y|^2+1)^{-q/2}(|x-z|^2+1)^{-q/2}dx,\\
I_{yz}^3:=&\int_{|x-y|,|x-z|>\frac12|y-z|\atop|x-c_{yz}|\le|y-z|}(|x-y|^2+1)^{-q/2}(|x-z|^2+1)^{-q/2}dx,\\
I_{yz}^4:=&\int_{|x-c_{yz}|>|y-z|}(|x-y|^2+1)^{-q/2}(|x-z|^2+1)^{-q/2}dx,
\end{align*}
Using that $|x-y|\le\frac12|y-z|$ implies $|x-z|\ge\frac12|y-z|$, we may estimate
\begin{align*}
I_{yz}^1&\le(|y-z|^2/4+1)^{-q/2}\int_{|x-y|\le\frac12|y-z|}(|x-y|^2+1)^{-q/2}dx\lesssim(|y-z|/2+1)^{d-2q},
\end{align*}
and similarly for $I_{yz}^2$. Moreover,
\begin{align*}
I_{yz}^3&\le (|y-z|^2/4+1)^{-q}\int_{|x-c_{yz}|\le|y-z|}dx\lesssim (|y-z|^2/4+1)^{-q}|y-z|^d\lesssim (|y-z|/2+1)^{d-2q},
\end{align*}
and also, since $d-2q<0$ follows from the choice $s> d-2$, $s\ge0$,
\begin{align*}
I_{yz}^4&\lesssim \int_{|x-c_{yz}|>|y-z|}(|x-y|+1)^{-q}(|x-z|+1)^{-q}dx\\
&\le \int_{|x-c_{yz}|>|y-z|}(|x-c_{yz}|-|y-z|/2+1)^{-2q}dx\lesssim (|y-z|/2+1)^{d-2q}.
\end{align*}
This proves, for all $y\ne z$,
\[\int_{\R^d}(|x-y|^2+1)^{-q/2}(|x-z|^2+1)^{-q/2}dx\lesssim (|y-z|/2+1)^{d-2q},\]
and hence by scaling
\begin{align*}
\int_{\R^d}|(x-y,\xi)|^{-q}|(x-z,\xi)|^{-q}dx\lesssim (|y-z|/2+|\xi|)^{d-2q},
\end{align*}
so that we obtain, as by definition $\gamma=q-d\in(-1,1)$,
\begin{align*}
\int_{\R^d\times\R}|\xi|^\gamma|(x-y,\xi)|^{-q}|(x-z,\xi)|^{-q}dxd\xi&\lesssim\int_{\R}|\xi|^{q-d}(|y-z|+|\xi|)^{d-2q}d\xi,
\end{align*}
Splitting the integrals over $\xi$ into the part where $|\xi|\le |y-z|$ and that where $|\xi|>|y-z|$, and noting that $q>1$ follows from $s>0$, we find
\begin{align*}
&\int_{\R^d\times\R}|\xi|^\gamma|(x-y,\xi)|^{-q}|(x-z,\xi)|^{-q}dxd\xi\\
\lesssim~&|y-z|^{d-2q}\int_{|\xi|\le|y-z|}|\xi|^{q-d}d\xi+\int_{|\xi|>|y-z|}|\xi|^{q-d}|\xi|^{d-2q}d\xi\lesssim|y-z|^{1-q}=|y-z|^{-s}.
\end{align*}
Combining this with~\eqref{eq:boundintegTN} finally yields
\[\int_{\R^d\times\R}|\xi|^\gamma|T_N^t|\lesssim\iint_{D^c}|y-z|^{-s}d(\mu_N^t+\mu^t)(y)d(\mu_N^t+\mu^t)(z),\]
and hence, by assumption~\eqref{eq:asnrjbounded}, since both the particle and the mean-field energies are decreasing along the flow (see Proposition~\ref{prop:existence} for the mean-field energy),
\begin{align*}
\int_{\R^d\times\R}|\xi|^\gamma|T_N^t|&\lesssim \int_{\R^d}\int_{\R^d}g_s(y-z)d\mu^t(y)d\mu^t(z)+\iint_{D^c}g_s(y-z)d\mu_N^t(y)d\mu_N^t(z)+2\int h^td\mu_N^t\\
&\le \int_{\R^d}\int_{\R^d}g_s(y-z)d\mu^\circ(y)d\mu^\circ(z)+\iint_{D^c}g_s(y-z)d\mu_N^\circ(y)d\mu_N^\circ(z)+2\|h^t\|_{\Ld^\infty}\lesssim_t1.
\end{align*}
We now briefly consider the case $s=0$, $d=1$ (hence $\gamma=0$, $q=1$). Let $1<p<\infty$ be such that $\|\nabla^2h^t\|_{\Ld^p}\lesssim_t1$. Arguing as above, we obtain
\[\int_{\R}|\nabla^2h^t(x)||(x-y,\xi)|^{-1}|(x-z,\xi)|^{-1}dx\lesssim \|\nabla^2h^t\|_{\Ld^\infty}(|y-z|+|\xi|)^{-1},\]
and similarly, by the Hölder inequality, for $1/p+1/p'=1$, $p'>1$,
\begin{align*}
\int_{\R}|\nabla^2 h^t(x)||(x-y,\xi)|^{-1}|(x-z,\xi)|^{-1}dx&\lesssim \|\nabla^2h^t\|_{\Ld^{p}}\bigg(\int_{\R}|(x-y,\xi)|^{-p'}|(x-z,\xi)|^{-p'}dx\bigg)^{1/p'}\\
&\lesssim \|\nabla^2h^t\|_{\Ld^{p}}(|y-z|+|\xi|)^{\frac1{p'}-2}.
\end{align*}
Splitting the integral over $\xi$ into the part where $|\xi|\le|y-z|\vee1$ and that where $|\xi|>|y-z|\vee1$, we may then estimate
\begin{align*}
&\int_{\R\times\R}|\nabla^2h^t(x)||(x-y,\xi)|^{-1}|(x-z,\xi)|^{-1}dxd\xi\\
\lesssim~~&\|\nabla^2h^t\|_{\Ld^\infty}\int_{|\xi|\le|y-z|\vee1}(|y-z|+|\xi|)^{-1}d\xi+\|\nabla^2h^t\|_{\Ld^{p}}\int_{|\xi|>|y-z|\vee1}(|y-z|+|\xi|)^{\frac1{p'}-2}d\xi\\
\lesssim_{t}~&1-0\wedge\log(|y-z|)=1+0\vee g_0(y-z),
\end{align*}
so that the conclusion now easily follows just as in the case $s>0$.

\bigskip
\noindent{\it Step~2: Approximation argument.}
For all $t\ge0$ and all $R\in(0,1)$, applying~\cite[Proposition~9.6]{SS-book},
there exists a smooth approximation $ v^t$ of the function $\nabla h^t\in C^{0,1}(\R^d;\R^d)$ such that $v^t$ is constant on each ball of the collection $\B_N^t(R)$, satisfies, for all $\alpha\in[0,1]$,
\begin{align}\label{eq:approxaffprop}
\|v^t-\nabla h^t\|_{C^\alpha}\le CR^{1-\alpha}\|\nabla^2h^t\|_{\Ld^\infty}\le C_tR^{1-\alpha},
\end{align}
and also satisfies $\|\nabla v^t\|_{\Ld^p}\lesssim_t1$ for some $p<\infty$ in the case $s=0$, $d=1$.
In this step, we prove the following estimate:
\begin{align}\label{eq:boundaff0}
\partial_t\Ec_N(t)\le-\int_{\R^d\times\R}|\xi|^\gamma\nabla v^t:T_N^t+C_t\,o_R(1),
\end{align}
where $o_R(1)$ denotes a quantity that goes to $0$ as $R\downarrow0$.

Using relation~\eqref{eq:SN-TN} as well as the integrability properties of $T_N^t$, we may decompose the first term in the right-hand side of~\eqref{eq:derENt0r} as follows:
\begin{align}\label{eq:rewriteDelort0}
\int_{\R^d\times\R}|\xi|^\gamma\nabla_{x,\xi}(\zeta(\xi)\nabla_xh^t(x)):T_N^t(x,\xi)dxd\xi=\langle S_N^t;\nabla h^t\rangle &=\langle S_N^t; v^t\rangle+\langle S_N^t;\nabla h^t- v^t\rangle\nonumber\\
&=\langle S_N^t\delta_{\R^d\times\{0\}}; v^t\rangle +\langle S_N^t;\nabla h^t- v^t\rangle \nonumber\\
&=\int_{\R^d\times\R}|\xi|^\gamma\nabla v^t: T_N^t+\langle S_N^t;\nabla h^t- v^t\rangle.
\end{align}
It remains to estimate the last term in the right-hand side of~\eqref{eq:rewriteDelort0}. Denoting for simplicity $w^t:=\nabla h^t- v^t$, we may decompose
\begin{align}
\langle S_N^t;\nabla h^t- v^t\rangle&=\iint_{D^c}(w^t(x)-w^t(y))\cdot \nabla g_s(x-y)d(\mu_N^t-\mu^t)(y)d(\mu_N^t-\mu^t)(x)\nonumber\\
&=2\int_{\R^d}\int_{\R^d} w^t(x)\cdot \nabla g_s(x-y)d\mu^t(y)d\mu^t(x)\label{eq:decompSNerr}\\
&\qquad-4\int_{\R^d}\int_{\R^d} (w^t(x)-w^t(y))\cdot \nabla g_s(x-y)d\mu_N^t(y)d\mu^t(x)\nonumber\\
&\qquad+2\iint_{D^c}w^t(x)\cdot \nabla g_s(x-y)d\mu_N^t(y)d\mu_N^t(x).\nonumber
\end{align}
For the first term in the right-hand side of~\eqref{eq:decompSNerr}, we simply have by~\eqref{eq:approxaffprop}
\begin{align*}
\bigg|\int_{\R^d}\int_{\R^d} w^t(x)\cdot \nabla g_s(x-y)d\mu^t(y)d\mu^t(x)\bigg|&=\bigg|\int_{\R^d}w^t\cdot\nabla h^td\mu^t\bigg|\le \|w^t\|_{\Ld^\infty}\|\nabla h^t\|_{\Ld^\infty}\le C_tR.
\end{align*}
As far as the second term is concerned, choosing $\sigma>s+1-d$, $0\le\sigma<1$, and recalling that $\mu^t$ remains bounded by assumption, we find by~\eqref{eq:approxaffprop}
\begin{align*}
&\bigg|\int_{\R^d}\int_{\R^d} (w^t(x)-w^t(y))\cdot \nabla g_s(x-y)d\mu_N^t(y)d\mu^t(x)\bigg|\\
\lesssim~&\|w^t\|_{C^\sigma}\sup_{y\in\R^d}\int |x-y|^{-s-1+\sigma}d\mu^t(x)\\
\le~&\|w^t\|_{C^\sigma}\sup_{y\in\R^d}\bigg(\|\mu^t\|_{\Ld^\infty}\int_{|x-y|\le1} |x-y|^{-s-1+\sigma}dx+\int_{|x-y|>1}d\mu^t(x)\bigg)\\
\le~&\|w^t\|_{C^\sigma}(1+\|\mu^t\|_{\Ld^\infty})\le C_tR^{1-\sigma}.
\end{align*}
Combining these two bounds with~\eqref{eq:decompSNerr}, using~\eqref{eq:approxaffprop} once again, we obtain, for $R\downarrow0$,
\begin{align*}
|\langle S_N^t;\nabla h^t- v^t\rangle|&\lesssim_t o_R(1)+R\int_{\R^d}\bigg|\int_{\R^d\setminus\{x\}}\nabla g_s(x-y)d\mu_N^t(y)\bigg|~d\mu_N^t(x)\\
&\lesssim_t o_R(1)+R\int_{\R^d}\bigg|\pv\int_{\R^d\setminus\{x\}}\nabla g_s(x-y)d(\mu_N^t-\mu^t)(y)\bigg|~d\mu_N^t(x)\\
&\qquad+R\int_{\R^d}\bigg|\pv\int_{\R^d\setminus\{x\}}\nabla g_s(x-y)d\mu^t(y)\bigg|~d\mu_N^t(x),
\end{align*}
and thus, noting that
\[\int_{\R^d}\bigg|\pv\int_{\R^d\setminus\{x\}}\nabla g_s(x-y)d\mu^t(y)\bigg|~d\mu_N^t(x)\le\|\nabla h^t\|_{\Ld^\infty}\le C_t,\]
we find
\begin{align*}
|\langle S_N^t;\nabla h^t- v^t\rangle|&\lesssim_t o_R(1)+R\int_{\R^d}\bigg|\pv\int_{\R^d\setminus\{x\}}\nabla g_s(x-y)d(\mu_N^t-\mu^t)(y)\bigg|~d\mu_N^t(x).
\end{align*}
Hence, for all $\e\in(0,1)$, using inequality $R|a|\le \e a^2+ (4\e)^{-1} R^2$, we obtain
\begin{align*}
|\langle S_N^t;\nabla h^t- v^t\rangle|&\lesssim_t \e^{-1} o_R(1)+\e \int_{\R^d}\bigg|\pv\int_{\R^d\setminus\{x\}}\nabla g_s(x-y)d(\mu_N^t-\mu^t)(y)\bigg|^2~d\mu_N^t(x),
\end{align*}
and the result~\eqref{eq:boundaff0} then follows from~\eqref{eq:rewriteDelort0} and~\eqref{eq:derENt0r}, choosing $\e>0$ small enough (depending on $t$).

\bigskip
\noindent{\it Step 3: Modification with $\eta$-approximations.}
In the definition~\eqref{eq:TNt} of $T_N^t$, the diagonal terms were excluded. In order to apply inequality $2|ab|\le a^2+b^2$ to $T_N^t$ as in the proof of Lemmas~\ref{lem:proofsmooth} and~\ref{lem:proofsmooth2}, we would need to add these diagonal terms explicitly. Moreover, $\eta$-approximations then become needed in order to avoid the divergence of the corresponding diagonal terms that will appear after application of the above-mentioned inequality. More precisely, we prove in this step
\begin{align}\label{eq:boundaff-20b}
\partial_t\Ec_N(t)&\lesssim_t\int_{(\R^d\setminus\B_N^t(R))\times\R}|\xi|^\gamma|\nabla (h_{N,\eta}^t-h^t)|^2\\
&\qquad+\frac1{N^2}\sum_{i=1}^N\int_{(\R^d\setminus\B_N^t(R))\times\R}|\xi|^\gamma|\nabla v^t(x)||\nabla g_s(x-x_{i,N}^t,\xi)|^2dxd\xi+o_{R}(1)+o_\eta^{(N,R)}(1).\nonumber
\end{align}

By the choice of $v^t$ to be constant on each ball of the collection $\B_N^t(R)$ and the bound on $\nabla v^t$, \eqref{eq:boundaff0} becomes
\begin{align}\label{eq:newbounddtENballs}
\partial_t\Ec_N(t)\lesssim_t\int_{(\R^d\setminus\B_N^t(R))\times\R}|\xi|^\gamma|\nabla v^t(x)||T_N^t|+o_R(1).
\end{align}
Denote for simplicity
\[H_N^t(x,\xi):=(h_N^t-h^t)(x,\xi),\qquad H_{N,\eta}^t(x,\xi):=(h_{N,\eta}^t-h^t)(x,\xi),\]
and define, for all $1\le k\le d+1$ and all $1\le l\le d$, $T_{N,\eta}^{t;k,d+1}=0$ and
\[T_{N,\eta}^{t;kl}(x,\xi):=2 \partial_kH_{N,\eta}(x,\xi) \partial_lH_{N,\eta}(x,\xi)-\delta_{kl}|\nabla H_{N,\eta}^t(x,\xi)|^2.\]
For all $x$ with $d(x,\{x_{i,N}^t\}_{i=1}^N)>\eta$, we note that
\begin{align}\label{eq:HNetaequiv}
\nabla H_{N,\eta}^t(x,\xi)=\frac1N\sum_{i=1}^N\nabla g_{s,\eta}(x-x_{i,N}^t)-\nabla h^t(x)=\frac1N\sum_{i=1}^N\nabla g_{s}(x-x_{i,N}^t)-\nabla h^t(x)=\nabla H_N^t(x,\xi).
\end{align}
Also noting that definition~\eqref{eq:defGs} may be rewritten as
\begin{align*}
\frac1{N^2}\sum_{i=1}^NG_s^{kl}(x,\xi;x_{i,N}^t,x_{i,N}^t)&=2\iint_{D}\partial_k g_s(x-y;\xi)\partial_l g_s(x-z;\xi)d(\mu_N^t-\mu^t)(y)d(\mu_N^t-\mu^t)(z)\\
&\qquad-\delta_{kl}\iint_{D}\nabla g_s(x-y;\xi)\cdot\nabla g_s(x-z;\xi)d(\mu_N^t-\mu^t)(y)d(\mu_N^t-\mu^t)(z),
\end{align*}
definition~\eqref{eq:TNt} yields
\begin{align*}
&T_N^{t;kl}(x,\xi)+\frac1{N^2}\sum_{i=1}^NG_s^{kl}(x,\xi;x_{i,N}^t,x_{i,N}^t)\\
=~&2\int_{\R^d}\int_{\R^d}\partial_k g_s(x-y;\xi)\partial_l g_s(x-z;\xi)d(\mu_N^t-\mu^t)(y)d(\mu_N^t-\mu^t)(z)\\
&\qquad-\delta_{kl}\int_{\R^d}\int_{\R^d}\nabla g_s(x-y;\xi)\cdot\nabla g_s(x-z;\xi)d(\mu_N^t-\mu^t)(y)d(\mu_N^t-\mu^t)(z)\\
=~&2\partial_kH_N^t(x,\xi)\partial_l H_N^t(x,\xi)-\delta_{kl}|\nabla H_N^t(x,\xi)|^2.
\end{align*}
Combining this with~\eqref{eq:HNetaequiv} yields, for all $1\le k\le d+1$, all $1\le l\le d$, and all $x$ with $d(x,\{x_{i,N}^t\}_{i=1}^N)>\eta$,
\begin{align*}
T_N^{t;kl}(x,\xi)+\frac1{N^2}\sum_{i=1}^NG_s^{kl}(x,\xi;x_{i,N}^t,x_{i,N}^t)&=2\partial_kH_{N,\eta}^t(x,\xi)\partial_l H_{N,\eta}^t(x,\xi)-\delta_{kl}|H_{N,\eta}^t(x,\xi)|^2=T_{N,\eta}^{t;kl}(x,\xi).
\end{align*}
From~\eqref{eq:newbounddtENballs}, we then deduce, for all $\eta>0$ small enough such that $\bigcup_{i=1}^NB(x_{i,N}^t,\eta)\subset\B_N^t(R)$,
\begin{align*}
\partial_t\Ec_N(t)\lesssim_t \int_{(\R^d\setminus\B_N^t(R))\times\R}|\xi|^\gamma|T_{N,\eta}^t|+\frac1{N^2}\sum_{i=1}^N\int_{(\R^d\setminus\B_N^t(R))\times\R}|\xi|^\gamma|\nabla v^t(x)||G_s(x,\xi;x_{i,N}^t,x_{i,N}^t)|+o_R(1).
\end{align*}
The result~\eqref{eq:boundaff-20b} then follows, using inequality $2|ab|\le a^2+b^2$ in the form of
\[|T_{N,\eta}^t|\lesssim |\nabla(h_{N,\eta}^t-h^t)|^2,\qquad|G_s(x,\xi;x_{i,N}^t,x_{i,N}^t)|\lesssim |\nabla g_s(x-x_{i,N}^t,\xi)|^2.\]

\bigskip
\noindent{\it Step 4: Conclusion.}
In this step, we show that
\begin{align}\label{eq:conclinequfond}
\partial_t\Ec_N(t)&\lesssim_t \Ec_N(t)+\bigg(\frac{g_s(\eta)}N-\int_{\B_N^t(R)\times\R}|\xi|^\gamma|\nabla h_{N,\eta}^t|^2\bigg)\\
&\hspace{3cm}+\frac1{N^2}\sum_{i=1}^Ng_s^+(d(x_{i,N}^t,\partial\B_N^t(R)))+o_R(1)+o_{\eta}^{(N,R)}(1).\nonumber
\end{align}
The statement of Proposition~\ref{prop:MFL} immediately follows from this inequality, with the suitable choice of $R=R_N^t$, together with a simple Gronwall argument.

By Lemma~\ref{lem:linkENeta0}, inequality~\eqref{eq:boundaff-20b} may be rewritten as follows:
\begin{align*}
\partial_t\Ec_N(t)&\lesssim_t \Ec_N(t)+\frac{g_s(\eta)}N-\int_{\B_N^t(R)\times\R}|\xi|^\gamma|\nabla (h_{N,\eta}^t-h^t)|^2\\
&\qquad+\frac1{N^2}\sum_{i=1}^N\int_{(\R^d\setminus\B_N^t(R))\times\R}|\xi|^\gamma|\nabla v^t(x)||\nabla g_s(x-x_{i,N}^t,\xi)|^2dxd\xi+o_{R}(1)+o_\eta^{(N,R)}(1),
\end{align*}
or equivalently, expanding the square,
\begin{align}\label{eq:lastineqstepconcl}
\partial_t\Ec_N(t)&\lesssim_t \Ec_N(t)+\frac{g_s(\eta)}N-\int_{\B_N^t(R)\times\R}|\xi|^\gamma|\nabla h_{N,\eta}^t|^2-\int_{\B_N^t(R)\times\R}|\xi|^\gamma|\nabla h^t|^2+2\int_{\B_N^t(R)\times\R}|\xi|^\gamma \nabla h_{N,\eta}^t\cdot\nabla h^t\nonumber\\
&\qquad+\frac1{N^2}\sum_{i=1}^N\int_{(\R^d\setminus\B_N^t(R))\times\R}|\xi|^\gamma|\nabla v^t(x)||\nabla g_s(x-x_{i,N}^t,\xi)|^{2}dxd\xi+o_R(1)+o_{\eta}^{(N,R)}(1).
\end{align}
The last term in the first line of~\eqref{eq:lastineqstepconcl} is easily estimated as follows, using the notation~\eqref{def:smallballs} for the union $\B_N^t(R)$ of small balls,
\begin{align}
\bigg|\int_{\B_N^t(R)\times\R}|\xi|^\gamma\nabla h_{N,\eta}^t\cdot\nabla h^t\bigg|&\lesssim\|\nabla h^t\|_{\Ld^\infty}\frac1N\sum_{i=1}^N\int_{\B_N^t(R)\times\R}|\xi|^\gamma|(x-x_{i,N}^t,\xi)|^{-s-1}dxd\xi\nonumber\\
&\lesssim_t\frac1N\sum_{i=1}^N\int_{\B_N^t(R)}|x-x_{i,N}^t|^{1-d}dx\nonumber\\
&\lesssim\sum_{m=1}^{M_N^t}\int_{|x|\le 2r_{m,N}^t}|x|^{1-d}dx\lesssim\sum_{m=1}^{M_N^t}r_{m,N}^t= R,\label{eq:estimcroise}
\end{align}
while the term in the second line of~\eqref{eq:lastineqstepconcl} is, in the case $s>0$,
\begin{align}\label{eq:bound2concl}
&\frac1{N^2}\sum_{i=1}^N\int_{(\R^d\setminus\B_N^t(R))\times\R}|\xi|^\gamma|\nabla v^t(x)||\nabla g_s(x-x_{i,N}^t,\xi)|^{2}dxd\xi\nonumber\\
\lesssim_t~&\frac1{N^2}\sum_{i=1}^N\int_{(\R^d\setminus\B_N^t(R))\times\R}|\xi|^\gamma|(x-x_{i,N}^t,\xi)|^{-2(s+1)}dxd\xi\nonumber\\
\lesssim~~&\frac1{N^2}\sum_{i=1}^N\int_{\R^d\setminus\B_N^t(R)}|x-x_{i,N}^t|^{-d-s}dx\lesssim\frac1{N^2}\sum_{i=1}^Nd(x_{i,N}^t,\partial\B_N^t(R))^{-s}.
\end{align}
In the case $s=0$, $d=1$ (so $\gamma=0$), denoting $\rho_{i,N}^t:=d(x_{i,N}^t,\partial\B_N^t(R))$, and applying the Hölder inequality with $1/p+1/p'=1$, where $p<\infty$ is chosen in such a way that $\|\nabla v^t\|_{\Ld^p}\lesssim_t1$, the term in the second line of~\eqref{eq:lastineqstepconcl} is
\begin{align*}
&\frac1{N^2}\sum_{i=1}^N\int_{(\R\setminus\B_N^t(R))\times\R}|\nabla v^t(x)||\nabla g_0(x-x_{i,N}^t,\xi)|^{2}dxd\xi\\
\lesssim~~&\frac1{N^2}\sum_{i=1}^N\int_{\R\setminus\B_N^t(R)}|\nabla v^t(x)||x-x_{i,N}^t|^{-1} dx\\
\lesssim~~&\frac1{N^2}\sum_{i=1}^N\int_{\rho_{i,N}^t<|x-x_{i,N}^t|\le1}|\nabla v^t(x)||x-x_{i,N}^t|^{-1} dx+\frac1{N^2}\sum_{i=1}^N\int_{|x-x_{i,N}^t|>1}|\nabla v^t(x)||x-x_{i,N}^t|^{-1} dx\\
\lesssim_t~&\frac1{N^2}\sum_{i=1}^N\int_{\rho_{i,N}^t<|x-x_{i,N}^t|\le1}|x-x_{i,N}^t|^{-1} dx+\frac1{N^2}\sum_{i=1}^N\bigg(\int_{|x-x_{i,N}^t|>1}|x-x_{i,N}^t|^{-p'} dx\bigg)^{1/p'},
\end{align*}
and hence, by the choice $p'>1$,
\begin{align}\label{eq:bound2bisconcl}
\frac1{N^2}\sum_{i=1}^N\int_{(\R\setminus\B_N^t(R))\times\R}|\nabla v^t(x)||\nabla g_0(x-x_{i,N}^t,\xi)|^{2}dxd\xi&\lesssim_t\frac1{N^2}\sum_{i=1}^N(-\,0\wedge\log\rho_{i,N}^t)+N^{-1}.
\end{align}
The result~\eqref{eq:conclinequfond} then follows from inequality~\eqref{eq:lastineqstepconcl} together with~\eqref{eq:estimcroise} and with~\eqref{eq:bound2concl} or~\eqref{eq:bound2bisconcl}.
\end{proof}

\subsection{Bypass of tightness issues}
Assuming that $\Ec_N(0)\le o_N(1)$, Proposition~\ref{prop:MFL} yields $\Ec_N(t)\le C_t o_N(1)$ for all $t\in[0,T]$. If we know that the sequence $(\mu_N^t)_N$ is tight, Lemma~\ref{lem:propmodenergy} then allows us to conclude with the desired convergence $\mu_N^t\cvf*\mu^t$, while tightness can easily be checked under the additional assumption that the initial measures $\mu_N^\circ$'s are well localized in the sense that $\limsup_N\int|x|^2d\mu_N^\circ<\infty$. However, in the spirit of~\cite[Section~4.3.5]{Serfaty-15}, the following refinement of Lemma~\ref{lem:propmodenergy} shows that much more information may be directly extracted from the fact that $\Ec_N(t)\le C_to_N(1)$, so that in particular tightness is obtained a posteriori without any additional assumption.


\begin{cor}\label{cor:convprop}
Let the assumptions of Proposition~\ref{prop:MFL} prevail. Also assume that $\Ec_N(0)\le o_N(1)$. Then, for all $t\in[0,T]$, we have $\nabla h_N^t\to\nabla h^t$ in $\Ld^p_{\loc}(\R^d;\Ld^2(\R,|\xi|^\gamma d\xi))$ for all $1\le p<2d/(s+d)$, and hence $\mu_N^t\cvf*\mu^t$. In particular, $(\mu_N^t)_N$ is tight and Lemma~\ref{lem:propmodenergy} then implies the convergence of the energy.
\end{cor}

\begin{proof}
By assumption, Proposition~\ref{prop:MFL} yields $\Ec_N(t)\lesssim_t o_N(1)$. We split the proof into three steps.

\medskip
\noindent{\it Step 1: Strong convergence outside small balls.}
In this step, we prove
\begin{align}\label{eq:convBRBN0}
\iint_{(\R^d\setminus \B_N^t)\times\R}|\xi|^{\gamma}|\nabla (h_{N}^t-h^t)|^2&\lesssim_t o_N(1),
\end{align}
and hence, for any $1\le p\le 2$, the Hölder inequality implies for all $R>0$
\begin{align}\label{eq:convBRBN}
\int_{B_R\setminus \B_N^t}\bigg(\int_\R|\xi|^{\gamma}|\nabla (h_{N}^t-h^t)|^2\bigg)^{p/2}&\lesssim R^{d(1-p/2)}\bigg(\iint_{(\R^d\setminus \B_N^t)\times\R}|\xi|^{\gamma}|\nabla (h_{N}^t-h^t)|^2\bigg)^{p/2}\lesssim_{R,t}o_N(1).
\end{align}

Applying Lemma~\ref{lem:linkENeta0} and expanding the square, we may decompose as follows the $\Ld^2(\R^d\times\R,|\xi|^\gamma dxd\xi)$-norm of $\nabla(h_{N,\eta}^t-h^t)$ outside the small balls $\B_N^t$:
\begin{align*}
\iint_{(\R^d\setminus\B_N^t)\times\R}|\xi|^\gamma|\nabla (h_{N,\eta}^t-h^t)|^2&=\Ec_{N,\eta}(t)-\iint_{\B_N^t\times\R}|\xi|^\gamma|\nabla (h_{N,\eta}^t-h^t)|^2\\
&=\Ec_{N}(t)+\frac{g_s(\eta)}N-\iint_{\B_N^t\times\R}|\xi|^\gamma|\nabla (h_{N,\eta}^t-h^t)|^2+o_{\eta}^{(N)}(1)\\
&=\Ec_{N}(t)+\frac{g_s(\eta)}N-\iint_{\B_N^t\times\R}|\xi|^\gamma|\nabla h_{N,\eta}^t|^2-\iint_{\B_N^t\times\R}|\xi|^\gamma|\nabla h^t|^2\\
&\qquad+2\iint_{\B_N^t\times\R}|\xi|^\gamma\nabla h_{N,\eta}^t\cdot \nabla h^t+o_{\eta}^{(N)}(1).
\end{align*}
Applying Proposition~\ref{prop:MFL} in the form of $\Ec_N(t)\lesssim_t o_N(1)$, and using assumption~\eqref{eq:cond1}, this turns into
\begin{align*}
\iint_{(\R^d\setminus\B_N^t)\times\R}|\xi|^\gamma|\nabla (h_{N,\eta}^t-h^t)|^2&\le2\iint_{\B_N^t\times\R}|\xi|^\gamma\nabla h_{N,\eta}^t\cdot \nabla h^t+C_to_N(1)+o_{\eta}^{(N)}(1).
\end{align*}
Now arguing just as in~\eqref{eq:estimcroise}, we find
\begin{align*}
\bigg|\iint_{\B_N^t\times\R}|\xi|^\gamma\nabla h_{N,\eta}^t\cdot \nabla h^t\bigg|&\lesssim_t R_N^t\lesssim_to_N(1),
\end{align*}
and hence
\begin{align*}
\iint_{(\R^d\setminus\B_N^t)\times\R}|\xi|^\gamma|\nabla (h_{N,\eta}^t-h^t)|^2&\lesssim_t o_N(1)+o_{\eta}^{(N)}(1).
\end{align*}
Passing to the limit $\eta\downarrow0$ in this inequality, and noting that $\nabla h_{N,\eta}^t\to\nabla h_N^t$ in the distributional sense, the result~\eqref{eq:convBRBN0} follows.

\bigskip
\noindent{\it Step~2: Neglecting the contribution inside small balls.}
The contribution inside the small balls $\B_N^t$ is of course infinite, since $\nabla h_N^t$ does not belong to $\Ld^2(\R^d\times\R,|\xi|^\gamma dxd\xi)$. However, we show that it is small in $\Ld^p_{\loc}(\R^d;\Ld^2(\R,|\xi|^\gamma d\xi))$ for $p$ small enough. More precisely, for any $1\le p<2d/(s+d)$, we show that we have for all $R>0$
\begin{align}\label{eq:convBRinBN}
\int_{B_R\cap\B_N^t}\bigg(\int_\R|\xi|^\gamma|\nabla (h_N^t-h^t)|^2\bigg)^{p/2}&\lesssim_{R,t}o_N(1).
\end{align}

Decomposing $\nabla h_N^t(x)=\frac1N\sum_{i=1}^N\nabla g_s(x-x_{i,N}^t,\xi)$, the triangle inequality yields
\begin{align*}
\bigg(\int_{B_R\cap\B_N^t}\bigg(\int_\R|\xi|^\gamma|\nabla h_N^t|^2\bigg)^{p/2}\bigg)^{1/p}&\lesssim\frac1N\sum_{i=1}^N\bigg(\int_{B_R}\bigg(\int_\R|\xi|^\gamma|(x-x_{i,N}^t,\xi)|^{-2(s+1)}d\xi\bigg)^{p/2}dx\bigg)^{1/p}.
\end{align*}
A direct computation of the integral over $\xi$ yields
\begin{align*}
\bigg(\int_{B_R\cap\B_N^t}\bigg(\int_\R|\xi|^\gamma|\nabla h_N^t|^2\bigg)^{p/2}\bigg)^{1/p}&\lesssim\frac1N\sum_{i=1}^N\bigg(\int_{B_R}|x-x_{i,N}^t|^{\frac p2(\gamma+1-2(s+1))}dx\bigg)^{1/p}.
\end{align*}
As for each $i$ the integral over $x\in B_R$ is clearly bounded above by the same integral over $x\in B_R(x_{i,N}^t)$, we obtain
\begin{align*}
\bigg(\int_{B_R\cap\B_N^t}\bigg(\int_\R|\xi|^\gamma|\nabla h_N^t|^2\bigg)^{p/2}\bigg)^{1/p}&\lesssim\bigg(\int_{B_R}|x|^{\frac p2(\gamma+1-2(s+1))}dx\bigg)^{1/p},
\end{align*}
and hence, for any $1\le p<2d/(s+d)$,
\begin{align*}
\int_{B_R\cap\B_N^t}\bigg(\int_\R|\xi|^\gamma|\nabla h_N^t|^2\bigg)^{p/2}&\lesssim\int_{B_R}|x|^{\frac p2(\gamma+1-2(s+1))}dx=\int_{B_R}|x|^{-(s+d)p/2}dx\lesssim_R1,
\end{align*}
Now, for any $1\le p<2d/(s+d)$, choosing any $p<q<2d/(s+d)$, the Hölder inequality yields
\begin{align*}
\int_{B_R\cap\B_N^t}\bigg(\int_\R|\xi|^\gamma|\nabla h_N^t|^2\bigg)^{p/2}&\le |\B_N^t|^{1-p/q}\bigg(\int_{B_R\cap\B_N^t}\bigg(\int_\R|\xi|^\gamma|\nabla h_N^t|^2\bigg)^{q/2}\bigg)^{p/q}\lesssim_{R}|\B_N^t|^{1-p/q}\lesssim_{t}o_N(1).
\end{align*}
The result~\eqref{eq:convBRinBN} follows from this and from the Hölder inequality in the form
\begin{align*}
\int_{B_R\cap\B_N^t}\bigg(\int_\R|\xi|^\gamma|\nabla h^t|^2\bigg)^{p/2}\le|\B_N^t|^{1-p/2}\bigg(\iint_{B_R\times\R}|\xi|^\gamma|\nabla h^t|^2\bigg)^{p/2}&\lesssim_{R,t} o_N(1).
\end{align*}

\bigskip
\noindent{\it Step~3: Conclusion.}

Combining~\eqref{eq:convBRBN} and~\eqref{eq:convBRinBN}, for any $1\le p<2d/(s+d)$, we may conclude, for all $R>0$,
\[\int_{B_R}\bigg(\int_\R|\xi|^\gamma|\nabla (h_N^t-h^t)|^2\bigg)^{p/2}\lesssim_{R,t}o_N(1).\]
This proves $\nabla h_N^t\to\nabla h^t$ in $\Ld^p_{\loc}(\R^d;\Ld^2(\R,|\xi|^\gamma d\xi))$ for any $1\le p<2d/(s+d)$. Applying the operator $-\Div(|\xi|^\gamma\cdot)$ to both sides, we deduce $\mu_N^t\to\mu^t$ in the distributional sense on $\R^d\times\R$, so the result is proven.
\end{proof}

\subsection{Ball construction}\label{chap:ball0}
In this section, we make the heuristics of Remark~\ref{rem:conditions} rigorous, showing that for $0\le s<1$ the collection $\B_N^t(R_N^t)$ can indeed be chosen with $R_N^t\to0$ in such a way that both conditions~\eqref{eq:cond1} and~\eqref{eq:cond2} are satisfied.

Let us first describe the construction that we will use for the collection $\B_N^t(R)$, $R>0$. This is precisely the same construction as the one used e.g. in~\cite[Chapter~4]{SS-book}, which was introduced by~\cite{Sandier-98,Jerrard-99} for the analysis of the Ginzburg-Landau vortices. We first consider $N$ disjoint small balls centered at the points $x_{i,N}^t$'s with equal radii (smaller than $\eta_N/2$), and we grow their radii by the same multiplicative factor. At some point during this growth process, two (or more) balls may become tangent to one another. We then merge them into a bigger ball: if tangent balls are of the form $B(a_i,r_i)$, we merge them into $B(\sum_ia_ir_i/\sum_ir_i,\sum_ir_i)$. If the resulting ball intersects other balls, we proceed to another merging, and so on, until all the balls are again disjoint. Then again we grow all the resulting radii by a multiplicative factor, etc., and we stop when the total radius $R$ is the one desired.

As we will see, condition~\eqref{eq:cond2} is easily checked directly from the construction above, so we may focus on the validity of condition~\eqref{eq:cond1}.
Hence we need to study integrals of the form $\int_{\B_N^t(R)\times\R}|\xi|^\gamma|\nabla h_{N,\eta}^t|^2$ for $R>0$. For that purpose, the basic tool is then the following crucial lower bound, which is a refinement of~\cite[Lemma~2.2]{Petrache-Serfaty-14}. In the sequel, for $x\in\R^d$ and $t>0$, we denote by $B'(x,t)$ the ball of radius $t$ centered at $(x,0)$ in $\R^d\times\R$, and we set $B'_t:=B'(0,t)$.

\begin{lem}[Embryo of a lower bound]\label{lem:comput}
Let $R>r$, let $(z_i)_{i=1}^k$ be a collection of points inside the ball $B_r$, and let $(z_{k+i})_{i=1}^{l}$ be a collection of points outside the ball $B_R$. Then
\begin{align}\label{eq:conjgs}
\int_{B_R'\setminus B_r'}|\xi|^\gamma\bigg|\sum_{i=1}^{k+l}\nabla g_s(x-z_i,\xi)\bigg|^2dxd\xi\ge k^2(g_s(r)-g_s(R)).
\end{align}
The same remains true if point charges are smeared out on small spheres around them, i.e. if $g_s$ is replaced by $g_{s,\eta}$ with $\eta<d(\{z_i\}_{i=1}^{k+l},B_R\setminus B_r)$. In particular, for any $z_1$, any $R>\eta>0$, and any collection $(z_i)_{i=2}^{1+l}$ of points outside the ball $B(z_1,R+\eta)$,
\begin{align}\label{eq:conjgs2}
\int_{B'(z_1,R)}|\xi|^\gamma\bigg|\sum_{i=1}^{1+l}\nabla g_{s,\eta}(x-z_i,\xi)\bigg|^2dxd\xi\ge g_s(\eta)-g_s(R).
\end{align}
\end{lem}

\begin{proof}
{\it Step 1: Explicit value of $c_{d,s}$.}
We claim that the normalization constant $c_{d,s}$ for the Riesz kernel $g_s$ is given by the following formula, in terms of the beta function $\Bd(a,b)=\Gamma(a)\Gamma(b)/\Gamma(a+b)$ and of the measure $\omega_{d-1}$ of the unit sphere of dimension $d-1$,
\begin{align}\label{eq:cds}
c_{d,s}&=s\omega_{d-1} \Bd\bigg(\frac{s+2-d}2,\frac d2\bigg).
\end{align}

Integrating the equality $-\Div(|\xi|^\gamma \nabla g_s)=\delta_0$ on the infinite cylinder $C_0:=B_1\times\R$ in $\R^d\times\R$, we find by integration by parts
\begin{align*}
-1=\int_{C_0}\Div(|\xi|^\gamma \nabla g_s)=\int_{\partial C_0}|\xi|^\gamma n\cdot\nabla g_s&=\int_{-\infty}^\infty\int_{\partial B_1}|\xi|^\gamma \partial_rg_s(u,\xi)d\sigma(u)d\xi\\
&=2\omega_{d-1}\int_{0}^\infty\xi^\gamma \partial_rg_s(1,\xi)d\xi.
\end{align*}
Since by definition $g_s(x,\xi)=c_{d,s}^{-1}(|x|^{2}+|\xi|^2)^{-s/2}$, computing the radial derivative yields
\begin{align*}
c_{d,s}&=2s\omega_{d-1}\int_{0}^\infty\xi^\gamma (1+\xi^2)^{-s/2-1}d\xi=s\omega_{d-1}\int_{0}^\infty\xi^{(\gamma-1)/2} (1+\xi)^{-s/2-1}d\xi.
\end{align*}
The result~\eqref{eq:cds} then easily follows using the formula $\Bd(a,b)=\int_0^\infty t^{a-1}(1+t)^{-a-b}dt$ for all $a,b>0$.

\bigskip
\noindent{\it Step 2: Conclusion.}\nopagebreak

For all $t>0$, denote by $B_t'$ the ball of radius $t$ in $\R^d\times\R$, and also $\mu_{k,l}:=\sum_{i=1}^{k+l}\delta_{z_i}$. We may then estimate by the Cauchy-Schwarz inequality
\begin{align*}
\int_{B'_R\setminus B'_r}|\xi|^\gamma|\nabla g_s\ast \mu_{k,l}|^2&=\int_r^Rdt\int_{\partial B_t'}|\xi|^\gamma|\nabla g_s\ast \mu_{k,l}|^2\\
&\ge\int_r^Rdt \bigg(\int_{\partial B_t'}|\xi|^\gamma\bigg)^{-1}\bigg(\int_{\partial B_t'}|\xi|^\gamma n\cdot \nabla g_s\ast \mu_{k,l}\bigg)^2,
\end{align*}
where an integration by parts yields, for all $r\le t\le R$,
\begin{align*}
\int_{\partial B_t'}|\xi|^\gamma n\cdot\nabla g_s\ast\mu_{k,l}&=\int_{B_t'}\Div(|\xi|^\gamma\nabla g_s\ast\mu_{k,l})=-\mu_{k,l}(B_t)=-k,
\end{align*}
and hence
\begin{align*}
\int_{B_R'\setminus B_r'}|\xi|^\gamma|\nabla g_s\ast \mu_{k,l}|^2&\ge k^2\int_r^Rdt \bigg(\int_{\partial B_t'}|\xi|^\gamma\bigg)^{-1}.
\end{align*}
In order to compute this last integral, we use spherical coordinates:
\begin{align*}
\int_{\partial B_t'}|\xi|^\gamma&=t^{s+1}\omega_{d-1}\int_0^\pi (\sin\theta)^{d-1}|\cos\theta|^\gamma d\theta=t^{s+1}\omega_{d-1}\int_{-1}^1 (1-u^2)^{(d-2)/2}|u|^\gamma du\\
&=t^{s+1}\omega_{d-1}\int_{0}^1 (1-u)^{(d-2)/2}u^{(\gamma-1)/2} du=t^{s+1}\omega_{d-1}\Bd\bigg(\frac{s+2-d}2,\frac d2\bigg),
\end{align*}
where the last equality follows from the formula $\Bd(a,b)=\int_0^1 t^{a-1}(1-t)^{b-1}dt$ for all $a,b>0$.
By Step~1, this last expression is nothing but $t^{s+1}{c_{d,s}}/s$, so that we may conclude
\begin{align*}
\int_{B_R'\setminus B_r'}|\xi|^\gamma|\nabla g_s\ast \mu_{k,l}|^2&\ge k^2\frac{s}{c_{d,s}}\int_r^Rt^{-s-1}dt=k^2\frac{r^{-s}-R^{-s}}{c_{d,s}}=k^2(g_s(r)-g_s(R)).\qedhere
\end{align*}
\end{proof}

With this result at hand, as in~\cite[Chapter~4]{SS-book}, we may now deduce the following lower bound estimate for the energy on the balls of the collection $\B_N^t(R)$. For logarithmic interactions (thus in particular for the Ginzburg-Landau vortices, as treated in~\cite[Chapter~4]{SS-book}), a particularly simple additive structure shows up, simplifying computations a lot; here we show that the same result still holds for all $s\le1$.

\begin{prop}[Lower bound]\label{prop:lowerbound}
Let $R>0$. If $s\le1$, then, for all $0<\eta<\eta_N\wedge (R/N)$,
\begin{align}\label{eq:LB}
\int_{\B_N^t(R)\times\R}|\xi|^\gamma|\nabla h^t_{N,\eta}|^2\ge\frac1N(g_s(\eta)-g_s(R/N)).
\end{align}
\end{prop}

\begin{proof}
We prove that, for all $R>0$, if $B(y,r)$ is a ball belonging to the collection $\B_N^t(R)$ and contains $n$ of the particles $x_{i,N}^t$'s, then
\begin{align}\label{eq:inductionLB}
\int_{B'(y,r)}|\xi|^\gamma|\nabla h_{N,\eta}^t|^2\ge\frac n{ N^{2}}(g_s(\eta)-g_s(R/N)).
\end{align}
The desired result~\eqref{eq:LB} indeed follows from summing the corresponding inequalities~\eqref{eq:inductionLB} associated with each ball $B(y,r)$ of the collection $\B_N^t(R)$, noting that $B'(y,r)\subset B(y,r)\times\R$. We prove~\eqref{eq:inductionLB} by induction: we first show that it holds when $B(y,r)$ contains only one particle $x_{i,N}^t$, and then that it is preserved through the growth process.

First, suppose that $B(y,r)$ is a ball of $\B_N^t(R)$ and contains only one particle $x_{i,N}^t$. By definition we must have $B(y,r)=B(x_{i,N}^t,r)$ and $x_{j,N}^t\notin B(y,r+\eta)$ for all $j\ne i$. Lemma~\ref{lem:comput} in the form of~\eqref{eq:conjgs2} then yields
\[\int_{B'(y,r)}|\xi|^\gamma|\nabla h_{N,\eta}^t|^2\ge \frac1{N^{2}}(g_s(\eta)-g_s(r)).\]
This proves~\eqref{eq:inductionLB} when $B(y,r)$ contains only one particle $x_{i,N}^t$.

Now we need to prove that~\eqref{eq:inductionLB} is preserved by the growth process, i.e. that it remains true through both expansion and merging of balls. On the one hand, suppose that, for some $R>0$, $B(y,r)$ is a ball of $\B_N^t(R)$ for which~\eqref{eq:inductionLB} holds, and suppose that $B(y,r)$ inflates into $B(y,\alpha r)$ without merging, when passing from $\B_N^t(R)$ to $\B_N^t(\alpha R)$, for some $\alpha>1$. Let $n$ denote the number of particles in $B(y,r)$. By definition, $B(y,\alpha r)$ contains the same number of particles, and the choice of $\eta$ small enough ensures that no particle may lie in the annulus $B(y,\alpha r+\eta)\setminus B(y,\alpha r)$. Hence, Lemma~\ref{lem:comput} in the form of~\eqref{eq:conjgs} yields
\begin{align*}
\int_{B'(y,\alpha r)}|\xi|^\gamma|\nabla h_{N,\eta}^t|^2&\ge \int_{B'(y, r)}|\xi|^\gamma|\nabla h_{N,\eta}^t|^2+\int_{B'(y,\alpha r)\setminus B'(y,r)}|\xi|^\gamma|\nabla h_{N,\eta}^t|^2\\
&\ge\frac n{N^{2}}(g_s(\eta)-g_s(R/N))+\frac {n^2}{N^{2}}(g_s(r)-g_s(\alpha r)).
\end{align*}
Since by definition $r=nR/N$, we find, by the choice $s\le 1$, with $g_s(R/N)-g_s(\alpha R/N)\ge0$,
\begin{align*}
\int_{B'(y,\alpha r)}|\xi|^\gamma|\nabla h_{N,\eta}^t|^2&\ge \frac n{N^{2}} (g_s(\eta)-g_s(R/N))+\frac{n^{2-s}}{N^{2}}(g_s(R/N)-g_s(\alpha R/N))\\
&\ge \frac n{N^{2}} (g_s(\eta)-g_s(R/N))+\frac{n}{N^{2}}(g_s(R/N)-g_s(\alpha R/N))\\
&\ge\frac n{N^{2}}(g_s(\eta)-g_s(\alpha R/N)),
\end{align*}
so that $B(y,\alpha r)$ also satisfies~\eqref{eq:inductionLB}.

On the other hand, suppose that $B(y_{i},r_i)$, $i=1,\ldots,k$, are $k$ disjoint balls of $\B_N^t(R^-)$ for some $R>0$, suppose that each of them satisfies~\eqref{eq:inductionLB}, and suppose that these balls are merged by the growth process into a larger ball $B(y,r)$, which is then disjoint of all other balls of the collection $\B_N^t(R)$. Denoting by $n_i$ the number of points in $B(y_i,r_i)$, we then find
\begin{align*}
\int_{B'(y,r)}|\xi|^\gamma|\nabla h_{N,\eta}^t|^2\ge \sum_{i=1}^k\int_{B'(y_i,r_i)}|\xi|^\gamma|\nabla h_{N,\eta}^t|^2&\ge \frac1{N^{2}}\bigg(\sum_{i=1}^kn_i\bigg)(g_s(\eta)-g_s(R/N)),
\end{align*}
so that $B(y, r)$ also satisfies~\eqref{eq:inductionLB}. This completes the proof.
\end{proof}

We are now in position to prove that both conditions~\eqref{eq:cond1} and~\eqref{eq:cond2} may be satisfied whenever $s<1$, thus finishing the proof of Theorem~\ref{th:mfl}.

\begin{cor}[Checking conditions~\eqref{eq:cond1} and~\eqref{eq:cond2}]
If $0\le s<1$ and if $\B_N^t(\cdot)$ is constructed as above, then conditions~\eqref{eq:cond1} and~\eqref{eq:cond2} are automatically satisfied for any choice $N^{-(1-s)/s}\ll R_N^t\ll1$ if $0<s<1$, and for any choice $e^{-No_N(1)}\lesssim_t R_N^t\ll1$ if $s=0$.
\end{cor}

\begin{proof}
On the one hand, Proposition~\ref{prop:lowerbound} gives
\[\lim_{\eta\downarrow0}\bigg(\int_{\B_N^t(R)\times\R}|\xi|^\gamma|\nabla h_{N,\eta}^t|^2-\frac1Ng_s(\eta)\bigg)\ge-\frac1Ng_s(R/N).\]
On the other hand, since by definition $\bigcup_{i=1}^NB(x_{i,N}^t,R/N)\subset \B_N^t(R)$, we may estimate
\[\frac1{N^2}\sum_{i=1}^Ng_s^+(d(x_{i,N}^t,\partial\B_N^t(R)))\lesssim\frac1Ng_s^+(R/N).\]
Therefore, both conditions~\eqref{eq:cond1} and~\eqref{eq:cond2} are satisfied if we choose $R_N^t$ such that $\frac1Ng_s^+(R_N^t/N)\ll1$, and the result follows.
\end{proof}

\subsubsection*{Acknowledgements}
The work of the author is supported by F.R.S.-FNRS (Belgian National Fund for Scientific Research) through a Research Fellowship.
The author would also like to warmly thank his PhD advisor Sylvia Serfaty, for suggesting this problem and for very helpful discussions.

\bigskip
\bibliographystyle{plain}
\bibliography{biblio}

\bigskip
{\small (Mitia Duerinckx) {\sc Université Libre de Bruxelles (ULB), Brussels, Belgium, \& MEPHYSTO team, Inria Lille--Nord Europe, Villeneuve d'Ascq, France, \& Laboratoire Jacques-Louis-Lions, Université Pierre et Marie Curie (UPMC), Paris, France}

{\it E-mail address:} mduerinc@ulb.ac.be
}

\end{document}